\documentclass[11pt, notitlepage]{article}
\usepackage{amssymb,amsmath,comment}
\usepackage{amsthm}
\catcode`\@=11 \@addtoreset{equation}{section}
\def\thesection{\arabic{section}}

\def\theequation{\thesection.\arabic{equation}}
\catcode`\@=12
\usepackage{colortbl}
\usepackage{hyperref}
\usepackage[mathscr]{eucal}
\usepackage{epsf}
\usepackage{esint}
\usepackage{a4wide}

\newcommand{\fa} {\forall}

\newcommand{\al} {\alpha}

\newcommand{\de} {\delta}

\newcommand{\N}{\mathbb{N}}
\newcommand{\R}{\mathbb{R}}

\newcommand{\De} {\Delta}

\newcommand{\noi} {\noindent}

\usepackage[all]{xy}
\catcode`\@=11
\def\theequation{\@arabic{\c@section}.\@arabic{\c@equation}}
\catcode`\@=12

\def\N{{I\!\!N}}

\newtheorem{Theorem}{Theorem}[section]
\newtheorem{Lemma}[Theorem]{Lemma}

\newtheorem{Remark}[Theorem]{Remark}
\newtheorem{Definition}[Theorem]{Definition}

\begin{document}

{\vspace{0.01in}}

\title{ Multiplicity results for mixed local-nonlocal equations with singular and critical exponential nonlinearity in $\mathbb{R}^2$}

\author{Sanjit Biswas\footnote{Department of Mathematics and Statistics, Indian Institute of Technology Kanpur, Kanpur-208016, Uttar Pradesh, India, Email: sanjitbiswas410@gmail.com }}

\maketitle

\begin{abstract}\noindent
In this article, we prove the existence of at least two positive weak solutions for a mixed local-nonlocal singular problem in the presence of critical exponential nonlinearity in dimension two. The novelty of this work is the inclusion of a variable singular exponent in the context of mixed operator and critical exponential nonlinearity in $\R^2$. Our approach is based on sub-solution super-solution technique, combined with variational methods.
\end{abstract}
\noi{Keywords: Mixed local-nonlocal equation, variable singularity, critical nonlinearity in dimension two, Ekeland's variational principle.}

\maketitle

\bigskip

\tableofcontents

\section{Introduction and main results}
Given a bounded domain $\Omega\subset\R^2$ with Lipschitz boundary, we consider the following semi-linear mixed local-nonlocal equation
\begin{align}\label{ME}
    \begin{cases}
        &\mathcal{M}_\epsilon u:=-\De u+\epsilon(-\De)^su=\frac{\lambda}{u^{\de(x)}}+h(x,u) \text{ in }\Omega,\nonumber\\
    &u>0 \text{ in }\Omega \text{ and } u=0 \text{ in }\R^2\setminus\Omega,\tag{$P^\epsilon_\lambda$}
    \end{cases}
\end{align}
where $s\in(0,1)$, $\epsilon>0$ is a constant, $\lambda>0$ is a parameter, and $\delta\in C(\overline{\Omega};(0,1))$, $h\in C^1(\overline{\Omega}\times\R;\R)$. Here, $\De$ denotes the usual Laplacian and the operator $(-\De)^s$ is known as fractional Laplacian and defined by
 \begin{equation*}
     (-\De)^su(x):=P.V.\int_{\R^2} \frac{(u(x)-u(y))}{|x-y|^{2+2s}}\;dy,
 \end{equation*}
 where $s\in(0,1)$ and P.V. stands for the principle value integral. Note that some definitions of fractional Laplacian include a normalization constant, but for simplicity, we omit it here. The equations involving mixed local and nonlocal operators have numerous applications in stochastic processes, image processing, biology, etc.; for more details, we refer to \cite{serena} and the references therein. The study of mixed operators has been a central focus of many researchers; for instance, see \cite{ Valdinoci23, BiagCCM, Mingioni24, ValdinociProc, Juha22, Valdinoci24, Valdinoci25} and the references therein. In this article, our main concern is to establish the existence of two positive solutions to (\ref{ME}) under some assumptions on the function $h.$ To this end, we define $$H(x,s)=\int_0^s h(x,t)\;dt\text{ for all }s>0,\text{ and for all }x\in\overline{\Omega},$$ and impose the following conditions on $h:\overline{\Omega}\times\R\to\R$:
\begin{enumerate}
    \item [$(H1)$] $h\in C^1(\overline{\Omega}\times\R;\R),\;h(x,0)=0$ and $h(x,\cdot):[0,\infty)\to\R$ is non-decreasing for every $x\in \overline{\Omega}$.
    \item [($H2$)] There exists $\alpha_0>0$ such that $$\lim_{s\to\infty}\frac{h(x,s)}{e^{\alpha s^2}}=\begin{cases}
        0, \text{ if }\alpha>\alpha_0,\\
        \infty, \text{ if }\alpha<\alpha_0,
    \end{cases}$$ 
     uniformly for $x\in \overline{\Omega}$.
    \item [($H3$)] $\limsup_{s\to 0}\frac{2H(x,s)}{s^2}<\lambda_1,$ where $\lambda_1$ is the first eigen value of the operator $-\De$ with Dirichlet's boundary condition in $\Omega$.
    \item [$(H4)$] There exists $M>0$ such that $H(x,s)\leq M(1+h(x,s))$ for all $s>0$ and for all $x\in\overline{\Omega}$.
    \item [$(H5)$] For every $\alpha>\alpha_0$, $$\lim_{s\to\infty}\frac{\frac{\partial h}{\partial s}(x,s)}{e^{\alpha s^2}}=0$$ uniformly for $x\in\overline{\Omega}$.
    \item [$(H6)$] There exists $\sigma\in [0,1)$ such that $\lim_{s\to\infty}\frac{h(x,s)se^{s^\sigma}}{e^{\alpha_0s^2}}=\infty$ uniformly for $x\in\overline{\Omega}.$
\end{enumerate}
 One can check that $h(x,s)=\mu se^{4\pi s^2}$, for $\mu<\lambda_1$ satisfies all the above assumptions with $\alpha_0=4\pi$ and any $\sigma\in [0,1)$. One prototype of the equation (\ref{ME}) is the following:
\begin{align*}
    \begin{cases}
        &\mathcal{M}_\epsilon u:=-\De u+\epsilon(-\De)^su=\frac{\lambda}{u^{\de(x)}}+\mu ue^{4\pi u^2} \text{ in }\Omega,\nonumber\\
    &u>0 \text{ in }\Omega \text{ and } u=0 \text{ in }\R^2\setminus\Omega.
    \end{cases}
\end{align*}
 To understand the context of our study, we review the existing literature on singular problems, which have been the focus of extensive research over the past few decades. We start with the purely local case:
\begin{align}\label{I1}
    \begin{cases}
        -\De u=\frac{f}{u^{\delta(x)}} \text{ in }\Omega,\\
        u>0 \text{ in }\Omega\text{ and }u=0\text{ in }\R^n\setminus\Omega.
    \end{cases}
\end{align}
This problem has been widely studied for both variable and constant exponent cases. For the positive constant $\delta$, Crandall-Rabinowitz-Tartar \cite{Crand} proved the existence of a classical solution to (\ref{I1}) under specific regularity assumptions on the non-negative function $f$ and the domain $\Omega.$ Lazer-Mckenna proved the existence result under some weaker assumption on $\Omega$ in \cite{Lazer} and showed that the solution belongs to $H^1(\Omega)$ if and only if $0<\delta<3$, but does not belong to $C^1(\overline{\Omega})$ for $\delta>1$. For some non-negative integrable function $f$, the existence and regularity results of weak solutions to (\ref{I1}) were further studied in \cite{LB}.
In the case of variable exponent, for positive $f\in L^m(\Omega)$ with $m\geq 1$, existence results are established in \cite{Jose16} and the associated quasi-linear equations has been studied in \cite{Alves20, KB21, PG} and the references therein.\\
The multiplicity-related questions for the equation
\begin{align}\label{I2}
    \begin{cases}
        -\De u=\frac{\lambda}{u^{\delta}}+h(x,u) \text{ in }\Omega,\\
        u>0 \text{ in }\Omega\text{ and }u=0\text{ in }\R^n\setminus\Omega
    \end{cases}
\end{align}
has been addressed by many authors. For the constant exponent case, in \cite{Haitao}, Y. Haitao proved that there exists $\Lambda^*>0$ such that equation (\ref{I2}) admits at least two solutions for every $\lambda\in(0,\Lambda^*)$, one solution for $\lambda=\Lambda^*$, and no solution for $\lambda>\Lambda^*$, provided $n\geq 3,$ $h(x,s)\equiv s^q$, $0<\delta<1<q\leq 2^*-1.$ This result is generalized for the nonlinear case by Giacomoni-Schindler-Peter in \cite{Jac}. In dimension two, the study of critical problem has been initiated by Adimurthi \cite{ADIPISA}, where some existence result is obtained. In dimension two, the multiplicity result to equation (\ref{I2}) is obtained by Adimurthi-Giacomoni \cite{ADICCM} for Moser-Trudinger type nonlinearity $h$ and $0<\delta<3.$ For more multiplicity-related results, we refer the reader to \cite{Sarika24, Srenadh14, Srenadh18} and the references therein.\\
In the purely non-local case, the equation
\begin{align}\label{I3}
    \begin{cases}
        (-\De)^s u=\frac{f}{u^{\delta(x)}} \text{ in }\Omega,\\
        u>0 \text{ in }\Omega\text{ and }u=0\text{ in }\R^n\setminus\Omega,
    \end{cases}
\end{align}
has been studied for both constant and variable exponent cases. For the constant $\delta\in(0,1)$, in \cite{Yang}, the author proved the existence of a unique classical solution for $f\equiv 1.$ When $\delta>0$, Sciunzi et al. analyzed the quasi-linear version of \ref{I3} and obtained existence, uniqueness and qualitative properties of weak solutions in \cite{Scuinzi17}. For the variable case, we refer the reader to \cite{GTCPAA}, where the authors have established existence and regularity results, provided $f$ is a non-negative function belonging to $L^m(\Omega)\setminus\{0\}$ for some $m\geq 1.$\\
 The nonlocal perturbed problem of the form
\begin{align}\label{I4}
    \begin{cases}
        (-\De)^s u=\frac{\lambda }{u^{\delta(x)}}+ h(x,u) \text{ in }\Omega,\\
        u>0 \text{ in }\Omega\text{ and }u=0\text{ in }\R^n\setminus\Omega.
    \end{cases}
\end{align}
was studied in \cite{Barrio15, ST16}. When $\delta$ is a constant and $\delta\in(0,1]$, Srenadh-Tuhina \cite{ST16} proved that there exists $\Lambda^*>0$ such that for every $\lambda\in(0,\Lambda^*)$, (\ref{I4}) admits at least two positive solutions, provided $h(x,s)\equiv s^{2^*_s-1},$ where $2^*_s=\frac{2n}{n-2s}$. The associated nonlinear case was further explored in \cite{sekhar19, Tuhina19} and the reference therein. When $\delta$ is a variable, in \cite{GTCPAA}, the authors studied the associated quasi-linear version of (\ref{I4}) and proved some existence and multiplicity results. For more related results, we refer to \cite{Jac19} and the references therein.\\
The mixed local-nonlocal equation of the form
\begin{align}\label{I5}
    \begin{cases}
       -\De u+(-\De)^s u=\frac{f}{u^{\delta(x)}}\text{ in }\Omega,\\
        u>0 \text{ in }\Omega\text{ and }u=0\text{ in }\R^n\setminus\Omega
    \end{cases}
\end{align}
has also been studied for both constant and variable exponent cases. For constant $\delta>0$, the existence and regularity results of weak solutions to (\ref{I5}) were obtained in \cite{Arora}, provided $0\leq f\in L^m(\Omega)\setminus\{0\}$ for some $m\geq 1.$ Garain and Ukhlov \cite{PU22} considered the quasilinear case and obtained some existence, uniqueness, regularity and symmetric property of the solution. For variable $\delta$, existence and regularity results have been discussed in \cite{PG24, Biroudi} and the references therein.\\
Further, when $\delta$ is constant, the perturbed problem of the form
\begin{align}\label{I7}
    \begin{cases}
       -\De u+(-\De)^s u=\frac{\lambda}{u^{\delta(x)}}+h(x,u)\text{ in }\Omega,\\
        u>0 \text{ in }\Omega\text{ and }u=0\text{ in }\R^n\setminus\Omega
    \end{cases}
\end{align}
is explored in \cite{PGJGA}, where the author showed that there exists $\Lambda>0$ such that for every $\lambda\in(0,\Lambda)$, \ref{I7} admits at least two solutions, provided the constant $\delta\in(0,1)$, and $h(x,s)=s^q\; (1<q<2^*-1)$. The quasi-linear version of (\ref{I7}) has been discussed in \cite{KB24}. The critical case was recently addressed by Biagi and Vecchi \cite{BiagiCV}, who considered the equation:
\begin{align}\label{I6}
    \begin{cases}
        -\De u+\epsilon(-\De)^s u=\frac{\lambda}{u^{\delta}}+u^{2^*-1}\text{ in }\Omega,\\
        u>0 \text{ in }\Omega\text{ and }u=0\text{ in }\R^n\setminus\Omega,
    \end{cases}
\end{align}
for $\delta\in(0,1)$, $n\geq 3$ and proved that there exist $\Lambda>0$ and $\epsilon_0>0$ such that for every $\epsilon\in(0,\epsilon_0)$ and $\lambda\in(0,\Lambda),$ (\ref{I6}) possesses at least two positive solutions in $X.$\\
The multiplicity phenomenon for variable exponent singular problems involving mixed local and nonlocal operators and critical nonlinearity in dimension two remains unknown, even for constant $\delta$. In this work, we address this gap by focusing on the variable exponent case. Our proof relies on the sub-solution super-solution technique and variational approach, following arguments from \cite{ADICCM, BiagiCV, Haitao}. We also believe that our results can be proved for subcritical nonlinearity $h$ as well, but we prefer to focus on critical nonlinearity due to its delicate nature. In this article, our main results can be stated as follows:
\begin{Theorem}\label{T1}
    Assume that $\delta:\overline{\Omega}\to(0,1)$ is a continuous function and $h$ satisfies (H1)-(H5). Then, for every $\epsilon>0$, there exists $\Lambda_\epsilon>0$ such that 
    \begin{enumerate}
        \item [(i)] for every $\lambda\in(0,\Lambda_\epsilon]$, the problem (\ref{ME}) admits a positive solution in $X$ in the sense of Definition \ref{def};
        \item[(ii)] for every $\lambda>\Lambda_\epsilon$, the problem (\ref{ME}) has no solution.
    \end{enumerate}
    Furthermore, if $\lambda\in(0,\Lambda_\epsilon)$, then the above solution is a local minimizer of the energy functional 
    \begin{align}\label{energy}
    E_{\lambda,\epsilon}(u)=\frac{1}{2}\int_\Omega |u|^2\;dx+\frac{\epsilon}{2}\int_{\R^2}\int_{\R^2}\frac{|u(x)-u(y)|^2}{|x-y|^{2+2s}}\;dx\;dy-\lambda\int_\Omega\frac{|u|^{1-\delta(x)}}{(1-\delta(x))}\;dx-\int_\Omega H(x,u)\;dx,
\end{align}
defined on $X.$    
\end{Theorem}

\begin{Theorem}\label{T2}
    Assume that $\delta:\overline{\Omega}\to(0,1)$ is a continuous function and $h$ satisfies (H1)-(H6). Then, there exists $\epsilon_1>0$ and $\Lambda>0$ such that for every $\epsilon\in(0,\epsilon_1)$ and $\lambda\in (0,\Lambda)$, the equation (\ref{ME}) admits at least two solutions $u_{\lambda,\epsilon}$ and $v_{\lambda,\epsilon}$ in $X$ with $0<u_{\lambda,\epsilon}<v_{\lambda,\epsilon}$ in $\Omega$.
\end{Theorem}
Before proceeding, we summarize the key ideas of the proofs of our theorems. Concerning the proof of Theorem \ref{T1}, we analyze the behavior of the corresponding energy functional $E_{\lambda,\epsilon}$ in a small neighborhood around the origin and establish an existence of local minimizer, which leads us to the existence of $\Lambda_\epsilon.$ Finally, the existence result is proved by sub and supersolution technique, as in \cite{ADICCM, BiagiCV, Haitao}. To find the second solution, we basically use Ekeland's variational principle. A perturbation of the first solution (the local minimizer) with the Moser functions plays an important role to estimate the energy level and for this purpose, small $\epsilon>0$ is chosen to neglect the effect of nonlocal term.

The paper is organized in the following way. Section \ref{pre} is devoted to the functional settings of the problem. In section \ref{Aux}, some useful auxiliary results and their proofs are presented. Section \ref{PT1} contains proof of Theorem \ref{T1}. Section \ref{preT1} concerns some auxiliary results relevant to Theorem \ref{T2} and section \ref{PT2} contains proof of Theorem \ref{T2}. \\
\textbf{Notation:} Throughout this paper, unless otherwise mentioned, we will make use of the following notations and assumptions:
\begin{itemize}
    \item $\Omega\subset\R^2$ is a bounded domain with Lipschitz boundary.
    \item $L^2(\Omega):=\{u:\Omega\to \R \text{ measurable function }: \int_\Omega |u|^2\;dx<\infty\}$
    \item $C,\; C_1,\; C_2$ denote a positive constants, whose values may change from line to line or even in the same line.
    \item For a measurable set $B\subset\R^2$, $|B|$ denotes the Lebesgue measure of the set $B$. Moreover, for a function $u:B\to\R$, we define $u^+:=\max\{u,0\}$ and $u^-:=\min\{u,0\}.$
    \item Suppose $f,\;g:B\to \R$ are two functions. By $f\leq g$ in $B$, we mean that $f\leq g$ a.e. in $B.$ 
    \item By $a_n=O(b_n)$, we mean that $a_n\leq Cb_n$ for some constant $C>0.$
    \item By $a_n=o(1)$, we mean $a_n\to 0$ as $n\to\infty.$
\end{itemize}
\section{Preliminaries}\label{pre}
The Sobolev space $H^1(\Omega)$, is defined by
$$H^1(\Omega):=\{u\in L^2(\Omega): \nabla u\in L^2(\Omega,\R^2)\},$$
where $\nabla u$ denotes the weak gradient of $u.$
This space forms a Hilbert space equipped with norm $$\|u\|_{H^1(\Omega)}=\left(\int_\Omega \left(|u|^2+|\nabla u|^2\right)\;dx\right)^\frac{1}{2}.$$
The fractional Sobolev space $H^s(\Omega)$ for $0<s<1$, is defined by
$$H^s(\Omega):= \{u\in L^2(\Omega): \int_\Omega\int_\Omega\frac{|u(x)-u(y)|^2}{|x-y|^{2+2s}}\;dx\;dy<\infty\}$$
under the norm $$\|u\|_{H^s(\Omega)}=\left(\int_\Omega |u|^2\;dx+\int_\Omega\int_\Omega\frac{|u(x)-u(y)|^2}{|x-y|^{2+2s}}\;dx\;dy\right)^\frac{1}{2}.$$

For more details about the spaces, we refer to \cite{LC, Hitchhikersguide} and the references therein. The lemma, stated below, can be found in \cite[Proposition $2.2$]{Hitchhikersguide}.
\begin{Lemma}\label{l1}
Let $0<s<1$. Then there exists a positive constant $C=C(s,\Omega)$ such that 
$$\|u\|_{H^s(\Omega)}\leq C\|u\|_{H^1(\Omega)}$$
for every $u\in H^1(\Omega)$.
\end{Lemma}

 Involvement of local and nonlocal operators in the equation leads us to consider the following space
$$X:=\{ u\in H^1(\R^2): u=0 \text{ in }\R^2\setminus\Omega\}$$
under the norm
$$N_\epsilon(u)=\left(\int_\Omega |\nabla u|^2\;dx+\epsilon\int_{\R^2}\int_{\R^2}\frac{|u(x)-u(y)|^2}{|x-y|^{2+2s}}\;dx\;dy\right)^\frac{1}{2}.$$
This space is a Hilbert space equipped with the following inner product 
$$B_\epsilon(u,v)=\int_\Omega \nabla u\cdot \nabla v\;dx+\epsilon\int_{\R^2}\int_{\R^2}\frac{(u(x)-u(y))(v(x)-v(y))}{|x-y|^{2+2s}}\;dx\;dy.$$
The following lemma can be found in \cite[Lemma $2.1$]{Silva}.
\begin{Lemma}\label{locnon1}
Let $0<s<1$. There exists a constant $C=C(s,\Omega)>0$ such that
\begin{equation}\label{locnonsem}
\int_{\mathbb{R}^2}\int_{\mathbb{R}^2}\frac{|u(x)-u(y)|^2}{|x-y|^{2+2s}}\,dx\,dy\leq C\int_{\Omega}|\nabla u|^2\,dx
\end{equation}
for every $u\in X$.
\end{Lemma}

Due to Lemma \ref{locnon1}, we conclude that the norm $N_\epsilon$ is equivalent to the norm $\|u\|=\left(\int_\Omega|\nabla u|^2\;dx\right )^\frac{1}{2}.$ The space $(X,\|\cdot\|)$ is embedded into an Orlicz space, see \cite{Tru67}. Moreover, the map $X\to L^1(\Omega)$, defined by $u\to e^{|u|^\nu}$ is continuous for every $\nu\in [1,2]$ and compact except for the case $\nu=2$, see \cite{Moser}.
The following embedding is useful (see \cite{LC}).
\begin{Lemma}\label{Emb}
    The space $X$ is compactly embedded into $L^p(\Omega)$ for all $1\leq p<\infty.$
\end{Lemma}

Next lemma can be found in \cite{PL}.

\begin{Lemma}\label{PL}
    Suppose $\{u_n\in X: \|u_n\|=1\}$ be a sequence such that $u_n\rightharpoonup u$ weakly in $X$. Then, for every $0<p<(1-\|u\|^2)^{-1}$, we have $$\sup_{n\in\N}\int_\Omega e^{4\pi pu_n^2}\;dx<\infty.$$  
\end{Lemma}

The following Theorem was established by Moser \cite{Moser}. 
\begin{Theorem}\label{Moser}
   For every $p>0$, the following map $u\to e^{pu^2}$ is continuous from $X$ to $L^1(\Omega)$. Moreover, $$\sup\{\beta>0: \sup_{\|u\|\leq 1}\int_\Omega e^{\beta u^2}\;dx<\infty\}=4\pi.$$
\end{Theorem}
Now, we are ready to define the notion of a weak solution to the equation (\ref{ME}).
\begin{Definition}\label{def}
    A function $u\in X$ is said to be a super-solution (or sub-solution) of the equation (\ref{ME}), if it satisfies the following properties:
    \begin{enumerate}
        \item [(a)] for every $\omega\Subset\Omega$, there exists $C(\omega)>0$ such that $u\geq C(\omega)$ in $\omega$, and
        \item[(b)] for every non-negative $\phi\in C^1_c(\Omega)$,
        \begin{align}\label{Weak1}
        B_\epsilon(u,\phi)\geq (\text{ or } \leq) \lambda\int_\Omega\frac{\phi}{u^{\delta(x)}}\;dx+\int_\Omega h(x,u)\phi\;dx.
        \end{align}
    \end{enumerate}
We say $u\in X$ is a solution of the equation (\ref{ME}), if it satisfies (a) and the inequality (\ref{Weak1}) holds for every $\phi\in C^1_c(\Omega)$ without any sign condition.
\end{Definition}

\begin{Remark}
We remark that our definition is well stated. Since $u,\;\phi\in X$, by Cauchy-Schwarz inequality, we have $|B_\epsilon(u,\phi)|\leq N_\epsilon(u)N_\epsilon(\phi)<\infty.$ From the above condition (a), one has
$|\int_\Omega\frac{\phi}{u^{\delta(x)}}\;dx|<\infty.$ Moreover, the condition (H2) guarantees that for $\al>\al_0$, there exists $C>0$ such that $h(x,u)\leq Ce^{\alpha u^2}$. Using H\"{o}lder inequality and Theorem \ref{Moser}, we get
\begin{align}\label{B1}
    \int_\Omega h(x,u)\phi\;dx\leq C\int_\Omega e^{\al u^2}\phi\;dx\leq \|\phi\|_{L^2(\Omega)}\left( \int_\Omega e^{2\alpha u^2}\;dx\right)^\frac{1}{2}<\infty.
\end{align}  
\end{Remark}
\begin{Lemma}\label{WSL}
    Suppose $u\in X$ is a solution of the equation (\ref{ME}). Then, for every $v\in X$, we have
   \begin{align}\label{weak2}
       B_\epsilon(u,v)=\lambda\int_\Omega\frac{v}{u^{\delta(x)}}\;dx+\int_\Omega h(x,u)v\;dx.
   \end{align}
\end{Lemma}
\begin{proof}
Assume $v\in X$ such that $\omega:=\mathrm{supp}(v)\Subset\Omega.$ Then, there exists an open set $\Omega_1$ such that $\omega\Subset\Omega_1\Subset\Omega$ and a sequence $\{\phi_n\}\subset C^1_c(\Omega_1)$ such that $\phi_n\to v$ strongly in $X$ and pointwise a.e. in $\Omega.$ Since $u\in X$ is a solution of (\ref{ME}), we have
    \begin{align}\label{B2}
       B_\epsilon(u,\phi_n)=\lambda\int_\Omega\frac{\phi_n}{u^{\delta(x)}}\;dx+\int_\Omega h(x,u)\phi_n\;dx.
   \end{align}
Now, since $u\geq C(\Omega_1)$ in $\Omega_1$, we have
\begin{align*}
    \int_{\Omega_1}\frac{|\phi_n-v|}{u^{\delta(x)}}\;dx\leq C\int_\Omega|\phi_n-v|\;dx\leq C\|\phi_n-v\|\to 0 \text{ as }n\to\infty.
\end{align*}
Thus, $$\lim_{n\to\infty}\int_\Omega\frac{\phi_n}{u^{\delta(x)}}\;dx=\int_\Omega\frac{v}{u^{\delta(x)}}\;dx.$$ Moreover, using the estimate (\ref{B1}) and Lemma \ref{Emb}, we obtain
$$\int_\Omega |h(x,u)(\phi_n-v)|\;dx\leq C\|\phi_n-v\|\to 0\text{ as }n\to\infty,$$
which implies
$$\lim_{n\to\infty}\int_\Omega h(x,u)\phi_n\;dx=\int_\Omega h(x,u)\phi\;dx.$$
We take $n\to\infty$ in (\ref{B2}) to obtain
\begin{align}\label{B3}
       B_\epsilon(u,v)=\lambda\int_\Omega\frac{v}{u^{\delta(x)}}\;dx+\int_\Omega h(x,u)v\;dx.
   \end{align}
   Suppose $0\leq v\in X$, then by \cite[Lemma A.1]{Hirano04}, there exists a sequence $\{v_n\}_{n\in\N}$ such that $\mathrm{supp}(v_n)\Subset\Omega$ , $0\leq v_1\leq v_2\leq....\leq v_n\leq v_{n+1}\leq..\leq v$ and $v_n\to v$ strongly in $X$ and pointwise a.e. in $\Omega.$ Since the identity (\ref{B3}) holds for each $v_n$, using Monotone convergence theorem, we have
\begin{align}\label{B4}
       B_\epsilon(u,v)=\lambda\int_\Omega\frac{v}{u^{\delta(x)}}\;dx+\int_\Omega h(x,u)v\;dx.
   \end{align}
For general $v\in X$, we apply (\ref{B4}) for $v^+,\;v^-$ and then use linearity to obtain the result.  
\end{proof}

\section{Auxiliary Results}\label{Aux}
The following results are very crucial for our arguments. To be more precise, these lemmas play an important role to prove Lemma \ref{CL2} in the next section.
\begin{Lemma}\label{UB}
    Suppose $\lambda_0>0$ be a real number and $u_{\lambda,\epsilon}\in X$ is a solution of the equation (\ref{ME}) for $0<\lambda<\lambda_0$. Then, there exists $r_1>0$, such that if $\|u_{\lambda,\epsilon}\|\leq r_1$, then $\|u_{\lambda,\epsilon}\|_{L^\infty(\Omega)}\leq C,$ where $C=C(\lambda_0,\Omega)>0$ is a constant independent of $\epsilon.$
\end{Lemma}
\begin{proof}
 For $k>1$, we define $S(k):=\{x\in\Omega: u_{\lambda,\epsilon}\geq k\}$. Incorporating $\phi=(u_{\lambda,\epsilon}-k)^+$ in the weak formulation of equation (\ref{ME}), we obtain
\begin{align}
    \int_\Omega \nabla u_{\lambda,\epsilon}\cdot\nabla \phi\;dx+&\epsilon \underbrace{\int_{\R^2}\int_{\R^2}\frac{(u_{\lambda,\epsilon}(x)-u_{\lambda,\epsilon}(y))(\phi(x)-\phi(y))}{|x-y|^{2+2s}}\;dx\;dy}_{\geq 0}\nonumber\\
    &=\lambda\int_\Omega \frac{\phi}{u_{\lambda,\epsilon}^{\delta(x)}}\;dx+\int_\Omega h(x,u_{\lambda,\epsilon})\phi\;dx,
\end{align}
which yields
\begin{align}\label{S1}
    \int_\Omega |\nabla\phi|^2\leq C\lambda\int_\Omega \phi\;dx+\int_\Omega h(x,u_{\lambda,\epsilon})\phi\;dx.
\end{align}
By the property (H2), for $\alpha>\alpha_0$ there exists $C>0$ such that $$h(x,s)\leq Ce^{\alpha s^2} \text{ for all }s\geq 0.$$ Utilizing this fact and H\"{o}lder inequality, (\ref{S1}) implies
\begin{align}
    \int_\Omega |\nabla\phi|^2\;dx&\leq C\|\phi\|_{L^p(\Omega)}|S(k)|^{1-\frac{1}{p}}+C\int_{\Omega}e^{\alpha u_{\lambda,\epsilon}^2}\phi\;dx\nonumber\\
    &\leq C\|\phi\|_{L^p(\Omega)}|S(k)|^{1-\frac{1}{p}}+C\|\phi\|_{L^p(\Omega)}|S(k)|^{1-\frac{1}{p}-\frac{1}{q}}\left(\int_{\Omega}e^{q\alpha u_{\lambda,\epsilon}^2}\;dx\right)^{\frac{1}{q}}\nonumber\\
    &\leq C\|\phi\|_{L^p(\Omega)}|S(k)|^{1-\frac{1}{p}-\frac{1}{q}}\left(|S(k)|^\frac{1}{q}+\left(\int_{\Omega}e^{q\alpha \|u_{\lambda,\epsilon}\|^2\left(\frac{u_{\lambda,\epsilon}}{\|u_{\lambda,\epsilon}\|}\right)^2}\;dx\right)^{\frac{1}{q}}\right)\nonumber\\
    &\leq C\|\phi\|_{L^p(\Omega)}|S(k)|^{1-\frac{1}{p}-\frac{1}{q}}\left(1+\left(\int_{\Omega}e^{q\alpha \|u_{\lambda,\epsilon}\|^2\left(\frac{u_{\lambda,\epsilon}}{\|u_{\lambda,\epsilon}\|}\right)^2}\;dx\right)^{\frac{1}{q}}\right),
\end{align}
where $p,q>1$ such that $\frac{1}{p}+\frac{1}{q}+\frac{1}{q}=1$ (to be determined later). Whenever $q\alpha\|u_{\lambda,\epsilon}\|^2\leq 4\pi$, i.e., $\|u_{\lambda,\epsilon}\|\leq r_1(:=\frac{4\pi}{q\alpha})$ then $$\int_{\Omega}e^{q\alpha \|u_{\lambda,\epsilon}\|^2\left(\frac{u_{\lambda,\epsilon}}{\|u_{\lambda,\epsilon}\|}\right)^2}\;dx<C, $$
where $C>0$ is independent of $\epsilon.$ Thus, we have 
\begin{align}
    \int_\Omega |\nabla\phi|^2\;dx&\leq C\|\phi\|_{L^p(\Omega)}|S(k)|^{1-\frac{1}{p}-\frac{1}{q}},
\end{align}
where $C>0$ is a constant independent of $\epsilon.$ By Lemma \ref{Emb}, one has
\begin{align}
    \int_\Omega |\phi|^p\;dx\leq C |S(k)|^{p(1-\frac{1}{p}-\frac{1}{q})}.
\end{align}
For $1\leq k\leq h,$ we have
\begin{align}
    (h-k)^p|S(h)|=\int_{S(h)}(h-k)^p\;dx\leq \int_{S(h)}(u_{\lambda,\epsilon}-k)^p\;dx&\leq \int_\Omega |(u_{\lambda,\epsilon}-k)^+|^p\;dx\nonumber\\
    &\leq C|S(k)|^{p(1-\frac{1}{p}-\frac{1}{q})},
\end{align}
which gives
\begin{align}
    |S(h)|\leq \frac{C}{(h-k)^p}|S(k)|^{p(1-\frac{1}{p}-\frac{1}{q})},
\end{align}
where $C>0$ is a constant independent of $\epsilon.$ One can choose $p,q>1$ such that $\frac{1}{p}+\frac{1}{q}+\frac{1}{q}=1$ and $p(1-\frac{1}{p}-\frac{1}{q})>1$ (in particular, take $p=5, q=\frac{5}{2}).$ The conclusion follows from \cite[Lemma B.1]{Stampaccia00}.
\end{proof}
Suppose $w_{\lambda,\epsilon}\in X$ is the solution of purely singular problem:
\begin{align}\label{San}
    \begin{cases}
        \mathcal{M}_\epsilon u=\frac{\lambda}{u^{\delta(x)}} \text{ in }\Omega,\\
    u>0\text{ in }\Omega \text{ and }u=0 \text{ in }\R^2\setminus\Omega.
    \end{cases}
\end{align}

We have the following uniform positivity result.
\begin{Lemma}\label{LB}
Suppose $w_{\lambda,\epsilon}\in X$ is a solution to (\ref{San}) for $\lambda\in(0,\lambda_0)$. Then, there exists $\epsilon_0>0$ satisfying the following property: for every $\epsilon\in(0,\epsilon_0)$ and for any ball $B_{2r}(x_0)\Subset\Omega$, there exists a constant $\tau_0=\tau_0(\Omega,\lambda_0,r)>0$ independent of $\epsilon$ such that $w_{\lambda,\epsilon}\geq \tau_0$ in $B_r(x_0).$ 
\end{Lemma}
\begin{proof}
  We incorporate $w_{\lambda,\epsilon}$ in the weak formulation of (\ref{San}) and obtain
\begin{align}
    \int_\Omega |\nabla w_{\lambda,\epsilon}|^2\;dx\leq \lambda\int_\Omega w_{\lambda,\epsilon}^{1-\delta(x)}\;dx&\leq \lambda\left(\int_{\{w_{\lambda,\epsilon}\leq 1\}} w_{\lambda,\epsilon}^{1-\delta(x)}\;dx+\int_{\{w_{\lambda,\epsilon}\geq 1\}} w_{\lambda,\epsilon}^{1-\delta(x)}\;dx\right)\nonumber\\
    &\leq C\left( 1+\|w_{\lambda,\epsilon}\|\right).
\end{align}    
The rest of the proof similar to the same of \cite[Proposition 3.7]{BiagiCV}. We omit the proof here.
\end{proof}
Next lemma is useful to prove Theorem \ref{T2}. 
\begin{Lemma}\label{Comp}
    Suppose $u,\;v\in X$ are sub-solution and super-solution to equation (\ref{ME}), respectively. We suppose that $u\leq v$ in $\Omega$ and for every $\omega\Subset\Omega$, there exists $C(\omega)>0$ such that $u\geq C(\omega)$ in $\omega$. Then, either $u\equiv v$ in $\Omega$ or $u<v$ in $\Omega$.
\end{Lemma}
\begin{proof}
Let $\omega\Subset\Omega$ be an arbitrary open set.
Using (H1) and the mean value theorem, we have
\begin{align*}
    \mathcal{M}_\epsilon(v-u)&\geq \lambda\left(\frac{1}{v^{\delta(x)}}-\frac{1}{u^{\delta(x)}}\right)+h(x,v)-h(x,u)\text{ in }\omega \nonumber\\
    &\geq-\lambda\delta(x)\frac{(v-u)}{(u+\theta(v-u))^{1+\delta(x)}}\text{ in }\omega\nonumber\\
    &\geq-C(v-u) \text{ in }\omega,
\end{align*}
for some constant $C=C(\omega)>0$. Which reveals that
\begin{equation*}
    \mathcal{M}_\epsilon(v-u)+C(v-u)\geq 0 \text{ and }v-u\geq 0 \text{ in }\omega.
\end{equation*}
 Hence, the conclusion follows from \cite[Proposition $3.3$]{BiagiCV} together with standard covering argument.
\end{proof}

\section{Proof of Theorem \ref{T1}}\label{PT1}
We define $$\Lambda_\epsilon:=\sup\{\lambda\geq0: \text{ (\ref{ME}) has a solution in }X\}.$$
Throughout this section, we consider $0<\delta_0\leq\delta(x)\leq\delta_1<1$ for all $x\in\overline{\Omega}.$
\begin{Lemma}\label{Lm}
    The following conclusion holds: $0<\Lambda_\epsilon<\infty.$
\end{Lemma}
\begin{proof}
Firstly, we prove that $\Lambda_\epsilon<\infty.$ To this concern, we consider the eigen pair $(\mu_{1,\epsilon},\phi_{\epsilon})$ of the operator $\mathcal{M}_\epsilon$ with Dirichlet boundary condition, i.e.,
\begin{align}\label{EP}
    \begin{cases}
        \mathcal{M}_\epsilon \phi_\epsilon=\mu_{1,\epsilon}\phi_\epsilon \text{ in }\Omega,\\
        \phi_\epsilon>0 \text{ in }\Omega \text{ and } \phi_\epsilon=0 \text{ in }\R^2\setminus\Omega. 
    \end{cases}
\end{align}
We refer to \cite{Valdinoci23} for the existence of such a pair. Using (H2) and continuity of $\delta$, we can find $\lambda^*>0$ such that $\frac{\lambda^*}{s^{\delta(x)}}+h(x,s)\geq 2\mu_{1,\epsilon}s$ for all $s>0$ and for all $x\in\overline{\Omega}.$ Claim that for every $\lambda\geq\lambda^*$, the equation (\ref{ME}) does not have a solution. We prove it by contradiction. If possible, let $u\in X$ be a solution of the equation (\ref{ME}) for some $\lambda\geq\lambda^*$. By incorporating $\phi_\epsilon$ in (\ref{weak2}) we get
\begin{align}
    B_\epsilon(u,\phi_\epsilon)=\lambda\int_\Omega\frac{\phi_\epsilon}{u^{\delta(x)}}\;dx+\int_\Omega h(x,u)\phi_\epsilon\;dx,
\end{align}
which yields
\begin{align*}
 \mu_{1,\epsilon}\int_\Omega\phi_\epsilon u\;dx=\int_\Omega\left(\frac{\lambda}{u^{\delta(x)}}+ h(x,u)\right)\phi_\epsilon\;dx\geq \int_\Omega \left(\frac{\lambda^*}{u^{\delta(x)}}+h(x,u)\right )\phi_\epsilon\;dx\geq 2\mu_{1,\epsilon}\int_\Omega \phi_\epsilon u\;dx,
\end{align*}
which is a contradiction. Hence, $\Lambda_\epsilon<\infty.$ The rest of the proof aims to conclude $\Lambda_\epsilon>0.$ Consider the functional
$$E_{\lambda,\epsilon}(u)=\frac{1}{2}N_\epsilon^2(u)-\lambda\int_\Omega\frac{|u|^{1-\delta(x)}}{(1-\delta(x))}\;dx-\int_\Omega H(x,u)\;dx.$$
Applying (H1) and (H3), we have for $\alpha>\alpha_0,\;q>2$ there exists a constant $C>0$ such that $$H(x,u)\leq \frac{\mu}{2}u^2+C|u|^qe^{\alpha u^2},$$ where $\mu:=\limsup_{s\to 0}\frac{2H(x,s)}{s^2}<\lambda_1.$
Using this fact together with H\"{o}lder inequality and Lemma \ref{Emb}, we obtain
\begin{align}\label{A1}
    E_{\lambda,\epsilon}(u)&\geq \frac{1}{2}\int_\Omega |\nabla u|^2\;dx-\frac{\mu}{2}\int_\Omega u^2\;dx-\lambda\int_\Omega\frac{|u|^{1-\delta(x)}}{(1-\delta(x))}\;dx-C\int_\Omega|u|^qe^{\alpha u^2}\;dx\nonumber\\
    &\geq \frac{1}{2}(1-\frac{\mu}{\lambda_1})\|u\|^2-\lambda\int_\Omega\frac{|u|^{1-\delta(x)}}{(1-\delta(x))}\;dx-C\|u\|_{L^{2q}(\Omega)}^q\left(\int_\Omega e^{2\alpha u^2}\;dx\right)^\frac{1}{2}\nonumber\\
    &\geq \frac{1}{2}(1-\frac{\mu}{\lambda_1})\|u\|^2-\lambda\int_\Omega\frac{|u|^{1-\delta(x)}}{(1-\delta(x))}\;dx-C\|u\|^q\left(\int_\Omega e^{2\alpha\|u\|^2\left(\frac{u}{\|u\|}\right)^2}\;dx\right)^\frac{1}{2}.
\end{align}
Choose $\sigma>0$ such that $2\alpha\sigma^2\leq4\pi.$ Due to Theorem \ref{Moser}, for every $\|u\|\leq \sigma$, we have 
\begin{align}
     E_{\lambda,\epsilon}(u)\geq \frac{1}{2}(1-\frac{\mu}{\lambda_1})\|u\|^2-\lambda\int_\Omega\frac{|u|^{1-\delta(x)}}{(1-\delta(x))}\;dx-C\|u\|^q
\end{align}

We choose $R_0\in(0,\sigma)$, small enough such that $\frac{1}{2}(1-\frac{\mu}{\lambda_1})R_0^2-CR_0^q\geq \frac{1}{4}(1-\frac{\mu}{\lambda_1})R_0^2$. Now, for every $R\leq R_0$, we define $\Lambda:=\frac{\frac{1}{8}(1-\frac{\mu}{\lambda_1})R^2}{\sup_{\|u\|=R}\int_\Omega\frac{|u|^{1-\delta(x)}}{(1-\delta(x))}\;dx}$ (independent of $\epsilon$).
Hence, for every $\lambda\in(0,\Lambda)$, we have
\begin{align}\label{A22}
    E_{\lambda,\epsilon}(u)\geq\begin{cases}
        \frac{1}{8}(1-\frac{\mu}{\lambda_1})R^2, \text{ if }u\in\partial B(0,R)\\
        -\frac{3}{8}(1-\frac{\mu}{\lambda_1})R^2, \text{ if }u\in B(0,R).
    \end{cases}
\end{align}

Fix a non-negative function $\phi\in C^\infty_c(\Omega)$ such that $0\leq\phi\leq 1$ in $\Omega.$ Therefore, for $t>0$
\begin{align*}
    \frac{E_{\lambda,\epsilon}(t\phi)}{t}&=\frac{t}{2}N_\epsilon^2(\phi)-\lambda\int_\Omega \frac{\phi^{1-\delta(x)}}{t^{\delta(x)}(1-\delta(x))}\;dx-\frac{1}{t}\int_\Omega H(x,t\phi)\;dx\nonumber\\
    &\leq \frac{t}{2}N_\epsilon^2(\phi)-\frac{\lambda}{t^{\delta_0}}\int_\Omega \frac{\phi^{1-\delta(x)}}{(1-\delta(x))}\;dx.
\end{align*}
Thus, one can choose $t>0$ such that $t\phi\in B(0,R)$ and $E_{\lambda,\epsilon}(t\phi)<0.$ Consequently, $m_{\lambda,\epsilon}:=\inf_{B(0,R)}E_{\lambda,\epsilon}(u)<0.$  Suppose $\{u_n\}_{n\in\N}$ be a minimizing sequence, i.e., $E_{\lambda,\epsilon}(u_n)\to m_{\lambda,\epsilon}$ as $n\to\infty$. Due to Lemma \ref{Emb}, up to a subsequence $u_n\rightharpoonup u$ weakly in $X$, $u_n\to u$ a.e. in $\Omega$ and $u_n\to u$ in $L^1(\Omega).$ We claim that $$\int_\Omega H(x,u_n)\;dx\to\int_\Omega H(x,u)\;dx\text{ and }\int_\Omega\frac{|u_n|^{1-\delta(x)}}{(1-\delta(x))}\;dx\to \int_\Omega\frac{|u|^{1-\delta(x)}}{(1-\delta(x))}\;dx.$$
To this end, we first observe that for the above $\alpha>\alpha_0$
\begin{align}\label{A2}
    \int_\Omega |H(x,u_n)-H(x,u)|^2\;dx&\leq 2\int_\Omega |H(x,u_n)|^2\;dx+2\int_\Omega |H(x,u)|^2\;dx\nonumber\\
    &\leq C\int_\Omega e^{2\alpha\|u_n\|^2\left(\frac{u_n}{\|u_n\|}\right)^2}\;dx+C\int_\Omega e^{2\alpha\|u\|^2\left(\frac{u}{\|u\|}\right)^2}\;dx\leq C,
\end{align}
where $C>0$ is a constant independent of $n$. In the second inequality, we have used the fact $H(x,s)\leq Ce^{\alpha s^2}$ for all $s\geq 0.$ The first part of the above claim follows from Egoroff's theorem and (\ref{A2}). To prove the second one, we use the estimation
\begin{align}
\frac{|u_n|^{1-\delta(x)}}{(1-\delta(x))}\leq \frac{1}{1-\delta_1}\left( |u_n|^{1-\delta(x)}\chi_{\{u_n< 1\}}+|u_n|^{1-\delta(x)}\chi_{\{u_n\geq 1\}}\right)\leq C(1+|u_n|).
\end{align}
Since $\int_\Omega (1+|u_n|)\;dx\to\int_\Omega (1+|u|)\;dx$, we have $$\lim_{n\to\infty}\int_\Omega\frac{|u_n|^{1-\delta(x)}}{(1-\delta(x))}\;dx=\int_\Omega\frac{|u|^{1-\delta(x)}}{(1-\delta(x))}\;dx.$$
Now, since $u_n\rightharpoonup u$ weakly in $X$, we have
\begin{align}
    E_{\lambda,\epsilon}(u)\leq \lim E_{\lambda,\epsilon}(u_n)=m_{\lambda,\epsilon},
\end{align}
which gives $E_{\lambda,\epsilon}(u)=m_{\lambda,\epsilon}.$ Using the monotone property of $H(x,\cdot)$ in $(0,\infty)$, we get $E_{\lambda,\epsilon}(|u|)=m_{\lambda,\epsilon}.$ Hence, we can assume that $u\in B(0,R)$ is a non-negative local minimizer of the functional $E_{\lambda,\epsilon}.$ Now, we will show that $u$ is a solution of the equation (\ref{ME}). To this aim, for $0\leq \phi\in X$, we have
\begin{align}\label{A3}
    \liminf_{t\to 0^+}\frac{E_{\lambda,\epsilon}(u+t\phi)-E_{\lambda,\epsilon}(u)}{t}\geq 0,
\end{align}
which implies
\begin{align*}
    B_\epsilon(u,\phi)\geq 0,
\end{align*}
which reveals that $u\in X$ is a super-solution of $\mathcal{M}_\epsilon u=0$ in $\Omega.$ Applying \cite[Proposition 3.3]{BiagiCV} with $a(x)=0$, we have $u>0$ in $\Omega.$ For small $t>0$, there exist $t_0\in(0,t)$ such that
\begin{align*}
    \Big|\frac{(u+t\phi)^{1-\delta(x)}-u^{1-\delta(x)}}{t(1-\delta(x))}\Big|=|(u+t_0\phi)^{-\delta(x)}\phi|\leq C\phi\text{ in }\omega:=\mathrm{supp}(\phi).
\end{align*}
By taking $t\to 0$ in (\ref{A3}) and using Fatou's lemma, we deduce
\begin{align*}
    B_\epsilon(u,\phi)-\lambda\int_\Omega\frac{\phi}{u^{\delta(x)}}\;dx-\int_\Omega h(x,u)\phi\;dx\geq 0, \fa\phi\in X, \phi\geq 0.
\end{align*}
For any $\phi\in X$, we incorporate $(u+\eta\phi)^+$ as a test function in (\ref{A3}), then divide both sides by $\eta$ and take $\eta\to 0$ to obtain
\begin{align*}
   B_\epsilon(u,\phi)-\lambda\int_\Omega\frac{\phi}{u^{\delta(x)}}\;dx-\int_\Omega h(x,u)\phi\;dx\geq 0.
\end{align*}
Hence, $u\in X$ is a solution of the equation (\ref{ME}) for every $\lambda\in (0,\Lambda).$ Consequently, $\Lambda_\epsilon>0.$
\end{proof}


\begin{Lemma}\label{L2}
    Suppose all the assumptions of Theorem \ref{T1} are satisfied and $\epsilon>0$. Then, for every $\lambda\in (0,\Lambda_\epsilon],$ the equation (\ref{ME}) admits a solution $u_{\lambda,\epsilon}\in X.$
\end{Lemma}
\begin{proof}
We split the proof into two cases.\\
\textbf{Case-I:} We assume $0<\lambda<\Lambda_\epsilon.$ In view of Perron's method, we consider the solutions $\underline{u}$ and $\overline{u}$ to the equations 
\begin{align}\label{A11}
    \begin{cases}
        \mathcal{M}_\epsilon u=\frac{\lambda}{u^{\delta(x)}} \text{ in }\Omega,\\
    u>0 \text{ in }\Omega \text{ and }u=0\text{ in }\R^2\setminus\Omega,
    \end{cases}
\end{align}
{and}
\begin{align}\label{RAN}
    \begin{cases}
        \mathcal{M}_\epsilon u=\frac{\overline{\lambda}}{u^{\delta(x)}}+h(x,u) \text{ in }\Omega,\\
        u>0 \text{ in }\Omega \text{ and } u=0\text{ in }\R^2\setminus\Omega,
    \end{cases}
\end{align}
respectively, where $\overline{\lambda}\in(\lambda,\Lambda_\epsilon).$ For the existence of $\underline{u}$, we refer to \cite{PG24}. It is immediate to show that $\underline{u}$ and $\overline{u}$ are sub and super solutions of the equation (\ref{ME}), respectively. Claim: $\underline{u}\leq\overline{u}$ in $\Omega.$ To this concern, define a non-decreasing smooth function $\psi:\R\to\R$ such that 
\begin{equation}
   \psi(t)= \begin{cases}
        1 &\text{ if } t\geq 1,\\
        0&\text{ if }t\leq 0.
    \end{cases}
\end{equation}
Let $\psi_\gamma(t)=\psi(\frac{t}{\gamma}),$ for $\gamma>0.$ Incorporating $\phi=\psi_\gamma(\underline{u}-\overline{u})$ in the weak formulations of (\ref{A11}) and (\ref{RAN}), then subtracting one from the other, we obtain
\begin{align}\label{Y1}
    B_\epsilon(\overline{u}-\underline{u}, \phi)&=\int_\Omega\left(\frac{\overline{\lambda}}{\overline{u}^{\delta(x)}}-\frac{\lambda}{\underline{u}^{\delta(x)}}\right)\phi\;dx+\int_\Omega h(x,\overline{u})\phi\;dx\nonumber\\
    &\geq \lambda\int_\Omega \left(\frac{1}{\overline{u}^{\delta(x)}}-\frac{1}{\underline{u}^{\delta(x)}}\right)\phi\;dx.
\end{align}
By using the monotone property of $\psi,$ (\ref{Y1}) yields
\begin{align*}
    \lambda\int_\Omega \left(\frac{1}{\overline{u}^{\delta(x)}}-\frac{1}{\underline{u}^{\delta(x)}}\right)\phi\;dx&\leq -\int_\Omega |\nabla(\overline{u}-\underline{u})|^2\psi'_\gamma(\underline{u}-\overline{u})\;dx\nonumber\\
    &+\int_{\R^2}\int_{\R^2}\frac{((\overline{u}-\underline{u})(x)-(\overline{u}-\underline{u})(y))(\phi(x)-\phi(y))}{|x-y|^{2+2s}}\;dx\;dy\nonumber\\
    &\leq 0,
\end{align*}
 which reveals
 \begin{align*}
     |\{x\in\Omega: \underline{u}>\overline{u}\}|=0.  
 \end{align*}
Consequently, $\underline{u}\leq \overline{u}$ in $\Omega.$
Now, we define $$\mathcal{M}:=\{u\in X: \underline{u}\leq u\leq \overline{u} \text{ in }\Omega\},$$ which is a closed and convex subset of $X$. Define $m:=\inf_{\mathcal{M}}E_{\lambda,\epsilon}(u).$ We claim that this infimum is achieved by a solution of the equation (\ref{ME}). To this end, we suppose $\{u_n\}_{n\in\N}$ be a sequence such that $E_{\lambda,\epsilon}(u_n)\to m$. That means, 
\begin{align*}
    \frac{1}{2}N_\epsilon^2(u_n)&=m+\lambda\int_\Omega\frac{|u_n|^{1-\delta(x)}}{(1-\delta(x))}\;dx+\int_\Omega H(x,u_n)\;dx+o(1)\\
    &\leq m+\lambda\int_\Omega\frac{|\overline{u}|^{1-\delta(x)}}{(1-\delta(x))}\;dx+\int_\Omega H(x,\overline{u})\;dx+o(1),
\end{align*}
which yields $\{u_n\}_{n\in\N}$ is bounded in $X.$ Due to Lemma \ref{Emb}, there exists $u\in X$ such that up to a subsequence $u_n\rightharpoonup u$ weakly in $X$, $u_n\to u$ a.e. in $\Omega$ and $u_n\to u$ strongly in $
L^1(\Omega).$ Since $u_n\leq \overline{u}$ in $\Omega$, by Lebesgue's dominated convergence theorem, we have
$$\int_\Omega H(x,u_n)\;dx\to\int_\Omega H(x,u)\;dx\text{ and }\int_\Omega\frac{|u_n|^{1-\delta(x)}}{(1-\delta(x))}\;dx\to \int_\Omega\frac{|u|^{1-\delta(x)}}{(1-\delta(x))}\;dx.$$
By using the weakly lower semi continuity of the norm function, we obtain
$$E_{\lambda,\epsilon}(u)\leq \lim E_{\lambda,\epsilon}(u_n)=m.$$ Moreover, since $u\in\mathcal{M}$, we get $E_{\lambda,\epsilon}(u)=m.$ The rest of the proof is devoted to conclude that $u$ solves the equation (\ref{ME}). Let $\phi\in C^\infty_c(\Omega)$. For small enough $\tau>0$, we define
$$\phi^\tau:=(u+\tau\phi-\overline{u})^+,\;\phi_\tau:=(u+\tau\phi-\underline{u})^-,$$ and $$w^\tau:=\begin{cases}
   \overline{u}, \text{ if } u+\tau\phi\geq\overline{u},\\
   u+\tau\phi, \text{ if } \underline{u}\leq u+\tau\phi\leq\overline{u},\\
   \underline{u}, \text{ if } u+\tau\phi\leq\underline{u}.
\end{cases}$$
Therefore, $w^\tau=u+\tau\phi-\phi^\tau+\phi_\tau\in\mathcal{M}.$ Since $\mathcal{M}$ is convex, one has
\begin{align*}
    \lim_{t\to 0^+}\frac{E_{\lambda,\epsilon}(u+t(w^\tau-u))-E_{\lambda,\epsilon}(u)}{t}\geq 0.
\end{align*}
By similar argument as in the previous theorem, we derive
\begin{align*}
    B_\epsilon(u,w^\tau-u)-\lambda\int_\Omega\frac{w^\tau-u}{u^{\delta(x)}}\;dx-\int_\Omega h(x,u)(w^\tau-u)\;dx\geq 0,
\end{align*}
which yields
\begin{align}\label{A4}
    B_\epsilon(u,\phi)-\lambda\int_\Omega\frac{\phi}{u^{\delta(x)}}-\int_\Omega h(x,u)\phi\;dx\geq\frac{1}{\tau}(R^\tau-R_\tau),
\end{align}
where $$R^\tau=B_\epsilon(u,\phi^\tau)-\lambda\int_\Omega\frac{\phi^\tau}{u^{\delta(x)}}-\int_\Omega h(x,u)\phi^\tau\;dx,$$ and $$R_\tau=B_\epsilon(u,\phi_\tau)-\lambda\int_\Omega\frac{\phi_\tau}{u^{\delta(x)}}-\int_\Omega h(x,u)\phi_\tau\;dx.$$
Now, since $\overline{u}$ is a super solution of (\ref{ME}) and (H1) holds, we have 
\begin{align}\label{A5}
    R^\tau&= B_\epsilon(u-\overline{u},\phi^\tau)+B_\epsilon(\overline{u},\phi^\tau)-\lambda\int_\Omega\frac{\phi^\tau}{u^{\delta(x)}}-\int_\Omega h(x,u)\phi^\tau\;dx\nonumber\\
    &\geq B_\epsilon(u-\overline{u},\phi^\tau)-\lambda\int_\Omega\left(\frac{1}{u^{\delta(x)}}-\frac{1}{\overline{u}^{\delta(x)}}\right )\phi^\tau\;dx+\int_\Omega \left(h(x,\overline{u})-h(x,u)\right )\phi^\tau\;dx\nonumber\\
    &\geq B_\epsilon(u-\overline{u},\phi^\tau)-\lambda\int_\Omega\left(\frac{1}{u^{\delta(x)}}-\frac{1}{\overline{u}^{\delta(x)}}\right )\phi^\tau\;dx\nonumber\\
    &= \int_{\Omega^\tau}|\nabla(u-\overline{u})|^2\;dx+\tau\int_{\Omega^\tau}\nabla(u-\overline{u})\cdot\nabla\phi\;dx+I_1-\lambda\int_{\Omega^\tau}\left(\frac{1}{u^{\delta(x)}}-\frac{1}{\overline{u}^{\delta(x)}}\right )(u-\overline{u})\;dx\nonumber\\
    &-\lambda\tau\int_{\Omega^\tau}\left(\frac{1}{u^{\delta(x)}}-\frac{1}{\overline{u}^{\delta(x)}}\right )\phi\;dx
    \nonumber\\
    &\geq \tau\int_{\Omega^\tau}\nabla(u-\overline{u})\cdot\nabla\phi\;dx+I_1-\lambda\int_{\Omega^\tau}\left(\frac{1}{u^{\delta(x)}}-\frac{1}{\overline{u}^{\delta(x)}}\right )(u-\overline{u})\;dx\nonumber\\
    &-\lambda\tau\int_{\Omega^\tau}\left(\frac{1}{u^{\delta(x)}}
    -\frac{1}{\overline{u}^{\delta(x)}}\right )\phi\;dx\nonumber\\
    &\geq \tau\int_{\Omega^\tau}\nabla(u-\overline{u})\cdot\nabla\phi\;dx+I_1-\lambda\tau\int_{\Omega^\tau}\left(\frac{1}{u^{\delta(x)}}
    -\frac{1}{\overline{u}^{\delta(x)}}\right )\phi\;dx,
\end{align}
where $\Omega_\tau=\mathrm{supp}(\phi^\tau)$ and $I_1=\epsilon\int_{\R^2}\int_{\R^2}\frac{((u-\overline{u})(x)-(u-\overline{u})(y))(\phi^\tau(x)-\phi^\tau(y))}{|x-y|^{2+2s}}\;dx\;dy.$ By using the similar estimation, as in \cite[Page 8 ]{Jac19}, we have $\frac{I_1}{\tau}\geq o(1)$ as $\tau\to 0.$ Moreover, since $|\Omega^\tau|\to 0$ as $\tau\to 0$, (\ref{A5}) yields
$\frac{R^\tau}{\tau}\geq o(1)$ as $\tau\to 0.$ Similarly, one can prove $\frac{R_\tau}{\tau}\leq o(1)$ as $\tau\to 0.$ Taking $\tau\to 0$ in (\ref{A4}) and using these facts, we obtain
\begin{align}
    B_\epsilon(u,\phi)-\lambda\int_\Omega\frac{\phi}{u^{\delta(x)}}-\int_\Omega h(x,u)\phi\;dx\geq 0, \fa \phi\in C^\infty_c(\Omega).
\end{align}
Hence, $u\in X$ is a solution of the equation (\ref{ME}).\\
\textbf{Case-II:} Suppose $\lambda=\Lambda_\epsilon$ and $\{\lambda_n\}$ be an increasing sequence such that $\lambda_n\to\lambda$ as $n\to\infty.$ Suppose $u_n\in X$ is the solution of (\ref{ME}) for $\lambda=\lambda_n$. Our next claim is that the sequence $\{u_n\}_{n\in\N}$ is bounded in $X.$ In order to prove this, we put $u_n$ in the weak formulation of (\ref{ME}) for $\lambda=\lambda_n$ and obtain
\begin{align}\label{A6}
    N_\epsilon^2(u_n)&=\lambda_n\int_\Omega u_n^{1-\delta(x)}\;dx+\int_\Omega h(x,u_n)u_n\;dx.
\end{align}
Furthermore,
\begin{align}\label{A7}
    E_{\lambda_n,\epsilon}(u_n)\leq E_{\lambda_n,\epsilon}(\underline{u}_n)&=\frac{1}{2}N_\epsilon^2(\underline{u}_n)-\lambda_n\int_\Omega\frac{\underline{u}_n^{1-\delta(x)}}{1-\delta(x)}\;dx-\int_\Omega H(x,\underline{u}_n)\;dx\nonumber\\
    &=\lambda_n\int_\Omega \left(\frac{1}{2}-\frac{1}{1-\delta(x)}\right)\underline{u}_n^{1-\delta(x)}\;dx-\int_\Omega H(x,\underline{u}_n)\;dx\nonumber\\
    &=-\frac{\lambda_n}{2}\int_\Omega \left(\frac{1+\delta(x)}{1-\delta(x)}\right)\underline{u}_n^{1-\delta(x)}\;dx-\int_\Omega H(x,\underline{u}_n)\;dx\leq 0,
\end{align}
where $\underline{u}_n$ is the sub-solution of (\ref{ME}) for $\lambda=\lambda_n$, satisfying (\ref{A11}). Thus, utilizing (\ref{A6}), and (\ref{A7}) we deduce
\begin{align}\label{A8}
    \frac{1}{2}\int_\Omega h(x,u_n)u_n\;dx-\int_\Omega H(x,u_n)\;dx\leq \frac{\lambda_n}{2}\int_\Omega\frac{1+\delta(x)}{1-\delta(x)}u_n^{1-\delta(x)}\;dx.
\end{align}
Using the property (H4), for every $\theta>2$, there exists $M_1>0$ such that 
\begin{equation}\label{G3}
    H(x,s)\leq \frac{1}{\theta}h(x,s)s+M_1\text{ for all } s\geq 0.
\end{equation}
Thus, we have 
\begin{align*}
    \left(\frac{1}{2}-\frac{1}{\theta}\right)\int_\Omega h(x,u_n)u_n\;dx\leq M_1|\Omega|+\frac{\lambda_n}{2}\int_\Omega\frac{1+\delta(x)}{1-\delta(x)}u_n^{1-\delta(x)}\;dx,
\end{align*}
which gives
\begin{align}\label{A9}
    \int_\Omega h(x,u_n)u_n\;dx\leq C\left( 1+\int_\Omega u_n^{1-\delta(x)}\;dx\right).
\end{align}
Combining (\ref{A6}) and (\ref{A9}), we infer 
\begin{align}
    N_\epsilon^2(u_n)&\leq C\left(1+\int_\Omega u_n^{1-\delta(x)}\;dx\right)\nonumber\\
    &\leq C\left(1+\int_{\{x\in\Omega:u_n\leq 1\}} |u_n|^{1-\delta(x)}\;dx+\int_{\{x\in\Omega:u_n\geq 1\}} |u_n|^{1-\delta(x)}\;dx\right)\nonumber\\
    &\leq C\left(1+\|u_n\|_{L^1(\Omega)}\right),
\end{align}
where $C>0$ is a constant independent of $n.$ This leads to conclude that $\{u_n\}_{n\in\N}$ is bounded in $X.$ By Lemma \ref{Emb}, there exists $u\in X$ such that $u_n\rightharpoonup u$ weakly $X$, $u_n\to u$ a.e. in $\Omega$ and $u_n\to u$ strongly in $L^p(\Omega)$ for all $1\leq p<\infty.$
Furthermore, we have $u_n\geq \underline{u}_1$ in $\Omega$. Since for every $\omega\Subset\Omega$, there exists $C(\omega)>0$ such that $\underline{u}_1\geq C(\omega)$ in $\omega$ so we have $u_n\geq C(\omega)$ in $\omega$. 
For every $\phi\in C^\infty_c(\Omega)$, one has 
\begin{align}\label{A10}
    B_\epsilon(u_n,\phi)=\lambda_n\int_\Omega \frac{\phi}{u_n^{\delta(x)}}\;dx+\int_\Omega h(x,u_n)\;\phi\;dx.
\end{align}
By Lebesgue's dominated convergence theorem, we have
$$\lim_{n\to\infty}\lambda_n\int_\Omega \frac{\phi}{u_n^{\delta(x)}}\;dx=\lambda\int_\Omega \frac{\phi}{u^{\delta(x)}}\;dx.$$
From the estimation (\ref{A9}), we have 
\begin{align}
    \int_\Omega|h(x,u_n)\phi|&= \int_{\{u_n\leq N\}} |h(x,u_n)\phi|\;dx+ \int_{\{u_n\geq N\}} |h(x,u_n)\phi|\;dx\nonumber\\
    &=  \int_{\{u_n\leq N\}} |h(x,u_n)\phi|\;dx+ \frac{1}{N}\int_{\{u_n\geq N\}} |h(x,u_n)u_n\phi|\;dx\nonumber\\
    &= \int_{\{u_n\leq N\}} |h(x,u_n)\phi|\;dx+ O(\frac{1}{N}).
\end{align}
 Thus, we get
$$\lim_{n\to\infty}\int_\Omega h(x,u_n)\phi\;dx=\int_\Omega h(x,u)\phi\;dx.$$
Taking $n\to\infty$ in (\ref{A10}) we obtain
\begin{align}
    B_\epsilon(u,\phi)=\lambda\int_\Omega \frac{\phi}{u^{\delta(x)}}\;dx+\int_\Omega h(x,u)\;\phi\;dx.
\end{align}
Consequently, $u$ is a solution of the equation (\ref{ME}) for $\lambda=\Lambda_\epsilon.$ This completes the proof.
\end{proof}
\begin{Lemma}\label{L3}
    Let $u$ be the solution of (\ref{ME}) for $\lambda\in(0,\Lambda_\epsilon)$, obtained in Lemma \ref{L2}. Then, $u$ is a local minimizer of the energy functional $E_{\lambda,\epsilon}.$
\end{Lemma}
\begin{proof}
    We prove it by contradiction. Suppose $u$ is not a local minimizer of $E_{\lambda,\epsilon}.$ Then, there exists a sequence $\{u_n\}_{n\in\N}\subset X$ such that $E_{\lambda,\epsilon}(u_n)<E_{\lambda,\epsilon}(u)$ and $\|u_n-u\|<\frac{1}{n}.$ Since  $\lambda\in(0,\Lambda_\epsilon),$ we assume $\underline{u},$ and $\overline{u}$ are the sub solution and super solution of (\ref{ME}) respectively, obtained in the proof of Lemma \ref{L2}. We define
    $\overline{v}_n:=(u_n-\overline{u})^+,\;\underline{v}_n:=(u_n-\underline{u})^-$ and $w_n:=\max\{\underline{u},\min\{u_n,\overline{u}\}\}=u_n-\overline{v}_n+\underline{v}_n.$ Our first claim is that $|\overline{S}_n|=|\underline{S}_n|=0,$ where $\overline{S}_n:=\mathrm{supp}(\overline{v}_n)$ and $\underline{S}_n:=\mathrm{supp}(\underline{v}_n).$ To this concern, for any $\xi>0,$ we choose $\Omega_1\Subset\Omega_2\Subset\Omega$ such that $|\Omega\setminus\Omega_1|<\frac{\xi}{2}.$ One has $u_n-u\geq \overline{u}-u\geq 0$ in $\overline{S}_n.$ Using (H1), we have
    \begin{align}
        \mathcal{M}_\epsilon(\overline{u}-u)\geq \lambda\left(\frac{1}{\overline{u}^{\delta(x)}}-\frac{1}{u^{\delta(x)}}\right)= -\frac{\delta(x)}{(u+t_0(\overline{u}-u))^{1+\delta(x)}}(\overline{u}-u) \text{ in }\Omega_2,
    \end{align}
    for some $t_0\in(0,t).$
    Since $u\leq \overline{u}$ in $\Omega$ and $u\geq C(\Omega_2)$ in $\Omega_2$, we have 
    \begin{align}
        \begin{cases}
            \mathcal{M}_\epsilon(\overline{u}-u)+a(x)(\overline{u}-u)\geq 0 \text{ in }\Omega_2,\\
        \overline{u}-u\geq 0 \text{ in } \R^2,
        \end{cases}
    \end{align}
    for some non-negative $a\in L^\infty(\Omega).$ By \cite[Proposition 3.3]{BiagiCV}, there exists $C=C(\Omega_1,u)>0$ such that $\overline{u}-u\geq C$ in $\Omega_1.$ Thus, we have
    \begin{align}\label{A13}
        |\overline{S}_n\cap\Omega_1|=\int_{\overline{S}_n\cap\Omega_1}\frac{|u_n-u|^2}{|u_n-u|^2}\;dx\leq  \int_{\overline{S}_n\cap\Omega_1}\frac{|u_n-u|^2}{|\overline{u}-u|^2}\;dx\leq \frac{1}{C^2}\int_{\Omega}|u_n-u|^2\;dx.
    \end{align}
Since $u_n\to u$ strongly in $X$, (\ref{A13}) implies that $|\overline{S}_n\cap\Omega_1|\to 0$ as $n\to\infty.$ Hence, there exists $N_0$ such that for all $n\geq N_0$, we have
$$|\overline{S}_n|\leq |\Omega\setminus\Omega_1|+|\overline{S}_n\cap\Omega_1|<\frac{\xi}{2}+\frac{\xi}{2}=\xi.$$ Thus, $|\overline{S}_n|\to0$ as $n\to\infty.$ Similarly, we can prove that $|\underline{S}_n|\to 0$ as $n\to\infty.$ Now,
\begin{align}
    \int_\Omega|\nabla\overline{v}_n|^2\;dx=\int_{\overline{S}_n}|\nabla(u_n-\overline{u})|^2\;dx&\leq 2\left(\int_{\overline{S}_n}|\nabla(u_n-u)|^2\;dx+\int_{\overline{S}_n}|\nabla(u-\overline{u})|^2\;dx\right)\nonumber\\
    &\leq 2\left(\|u_n-u\|^2+\int_{\overline{S}_n}|\nabla(u-\overline{u})|^2\;dx\right),
\end{align}
which yields that $\|\overline{v}_n\|\to0$ as $n\to\infty.$ Similarly, one can prove that $\|\underline{v}_n\|\to0$ as $n\to\infty.$ Along the lines of proof of \cite[ Inequality (4.20), page 29]{BiagiCV}, we obtain
\begin{align}\label{R0}
    E_{\lambda,\epsilon}(u_n)\geq E_{\lambda,\epsilon}(u)+R_n+T_n,
\end{align}
where 
\begin{align*}
    R_n=\frac{1}{2}N_\epsilon^2(\overline{v}_n)&+B_\epsilon(\overline{u},\overline{v}_n)-\lambda\int_{\overline{S}_n}\frac{(\overline{u}+\overline{v}_n)^{1-\delta(x)}-\overline{u}^{1-\delta(x)}}{1-\delta(x)}\;dx\\
    &-\int_{\overline{S}_n}(H(x,\overline{u}+\overline{v}_n)-H(x,\overline{u}))\;dx,
\end{align*}
and 
\begin{align*}
    T_n=\frac{1}{2}N_\epsilon^2(\underline{v}_n)&-B_\epsilon(\underline{u},\underline{v}_n)-\lambda\int_{\underline{S}_n}\frac{(\underline{u}-\underline{v}_n)^{1-\delta(x)}-\underline{u}^{1-\delta(x)}}{1-\delta(x)}\;dx\\
    &-\int_{\underline{S}_n}(H(x,\underline{u}-\underline{v}_n)-H(x,\underline{u}))\;dx.
\end{align*}

 Since $\overline{u}$ is a super solution of (\ref{ME}), we have
\begin{align}
    R_n&\geq \frac{1}{2}N_\epsilon^2(\overline{v}_n)+\lambda\int_{\overline{S}_n}\frac{\overline{v}_n}{\overline{u}^{\delta(x)}}\;dx-\lambda\int_{\overline{S}_n}\frac{(\overline{u}+\overline{v}_n)^{1-\delta(x)}-\overline{u}^{1-\delta(x)}}{1-\delta(x)}\;dx\nonumber\\
    &+\int_{\overline{S}_n}h(x,\overline{u})\overline{v}_n\;dx-\int_{\overline{S}_n}(H(x,\overline{u}+\overline{v}_n)-H(x,\overline{u}))\;dx.
\end{align}
Applying mean value theorem, we get
\begin{align}
    R_n&\geq \frac{1}{2}N_\epsilon^2(\overline{v}_n)+\underbrace{C_1\lambda\int_{\overline{S}_n}\frac{\overline{v}_n^2}{(\overline{u}+\theta_1\overline{v}_n)^{1+\delta(x)}}\;dx}_{\geq 0}-C_2\int_{\overline{S}_n} \frac{\partial h}{\partial s}(x,\overline{u}+\theta_2\overline{v}_n)\overline{v}_n^2\;dx,
\end{align}
for some $\theta_1,\theta_2\in(0,1)$ and constants $C_1,\;C_2>0$ independent of $n.$ Utilizing (H5), for $\alpha>\alpha_0$, we deduce
\begin{align}\label{R1}
    R_n&\geq C\|\overline{v}_n\|^2-C\int_{\overline{S}_n}e^{\alpha(\overline{u}+\theta_2\overline{v}_n)^2}\overline{v}_n^2\;dx\nonumber\\
    &\geq C\|\overline{v}_n\|^2-C\left(\int_{\overline{S}_n}e^{2\alpha(\overline{u}+\theta_2\overline{v}_n)^2}\;dx\right)^\frac{1}{2}\|\overline{v}_n\|_{L^4(\Omega)}^2\nonumber\\
    &\geq C\|\overline{v}_n\|^2\left(1- \left(\int_{\overline{S}_n}e^{4\alpha(\overline{u}^2+\overline{v}_n^2)}\;dx\right)^\frac{1}{2}\right)\nonumber\\
    &\geq C\|\overline{v}_n\|^2\left(1- \left(\int_{\overline{S}_n}e^{8\alpha\overline{u}^2}\;dx\right)^\frac{1}{4}\left(\int_{\overline{S}_n}e^{8\alpha\overline{v}_n^2}\;dx\right)^\frac{1}{4} \right)\nonumber\\
    &\geq C\|\overline{v}_n\|^2\left(1- \left(\int_{\overline{S}_n}e^{8\alpha\overline{u}^2}\;dx\right)^\frac{1}{4} \right),
\end{align}
for large $n$. Since $|\overline{S}_n|\to 0,$ (\ref{R1}) yields that $R_n\geq 0$ for large $n.$ Similar way, one can prove that $T_n\geq 0$ for large $n.$ 
Using this facts, (\ref{R0}) implies $E_{\lambda,\epsilon}(u_n)\geq E_{\lambda,\epsilon}(u)$ for large $n,$ which contradicts the fact $E_{\lambda,\epsilon}(u_n)<E_{\lambda,\epsilon}(u).$ This completes the proof. 
\end{proof}
\textbf{ Proof of Theorem \ref{T1}:} Proof of Theorem \ref{T1} follows from Lemma \ref{Lm}, Lemma \ref{L2} and Lemma \ref{L3}.

\section{Preliminaries for Theorem \ref{T2}}\label{preT1}

Throughout this section, we assume that all the assumptions of Theorem \ref{T2} are satisfied. Furthermore, in order to apply Lemma \ref{UB} and Lemma \ref{LB}, we make use of the following assumptions $\epsilon\in(0,\epsilon_0)$, $\lambda\in(0,\Lambda)$ and $r_2=\min\{r_1,R_0\},$ where $r_1,\epsilon_0$ and the pair $(R_0,\Lambda)$ are obtained in Lemma \ref{UB}, Lemma \ref{LB} and in the proof of Lemma \ref{Lm}, respectively. We suppose that  $u_{\lambda,\epsilon}\in B(0,r_2)$ is the solution of (\ref{ME}) obtained in Lemma \ref{Lm} for every $\lambda\in(0,\Lambda),$ which is a minimizer of the functional $E_{\lambda,\epsilon}$ in $B(0,r_2).$ We fix a ball $B_{2r}(x_0)\Subset\Omega$ and from Lemma \ref{UB}, Lemma \ref{LB} we have 
\begin{align}\label{bdd}
    \tau_0\leq u_{\lambda,\epsilon}\leq \tau_1 \text{ in }B_{r}(x_0),
\end{align}
where $\tau_0,\tau_1$  are independent of $\epsilon.$ \\
 For $n>1$, we define the Moser function 
\begin{equation*}
M_n(x):=\frac{1}{\sqrt{2\pi}}\begin{cases}
    (\mathrm{log} n)^\frac{1}{2}&\text{ if }0\leq|x-x_0|\leq \frac{r}{n},\\
    \frac{\mathrm{log}(\frac{r}{|x-x_0|})}{(\mathrm{log}(n))^\frac{1}{2}}&\text{ if }\frac{r}{n}\leq |x-x_0|\leq r,\\
    0&\text{ if }|x-x_0|\geq r.
\end{cases} 
\end{equation*}
One can show that $M_n\in X$ with $\|M_n\|=1$ and $\int_\Omega M_n^2\;dx=O(\frac{1}{\mathrm{log}(n)}).$
The following two lemmas play an very important role in the proof of Theorem \ref{T2}.
\begin{Lemma}\label{CL1}
    For large enough $n\in\N$, we have 
     $$E_{\lambda,\epsilon}(u_{\lambda,\epsilon}+tM_n)\to -\infty\text{ as }t\to\infty.$$
\end{Lemma}
\begin{proof}
    Due to the property (H2), for $\alpha<\alpha_0$ there exist $C_1,\;C_2>0$ such that $$H(x,s)\geq C_1e^{\alpha s^2}-C_2 \text{ for all } s\geq 0.$$ Thus, we have
    \begin{align}\label{D00}
        E_{\lambda,\epsilon}(u_{\lambda,\epsilon}+tM_n)&=\frac{1}{2}N_\epsilon^2(u_{\lambda,\epsilon}+tM_n)-\lambda\int_\Omega\frac{(u_{\lambda,\epsilon}+tM_n)^{1-\delta(x)}}{1-\delta(x)}\;dx\nonumber\\
        &-\int_\Omega H(x,u_{\lambda,\epsilon}+tM_n)\;dx\nonumber\\
        &\leq\frac{1}{2}N_\epsilon^2(u_{\lambda,\epsilon})+\frac{t^2}{2}N_\epsilon^2(M_n)+tB_\epsilon(u_{\lambda,\epsilon},M_n)-\lambda\int_\Omega\frac{(u_{\lambda,\epsilon}+tM_n)^{1-\delta(x)}}{1-\delta(x)}\;dx\nonumber\\
        &-C_1\int_\Omega e^{\alpha (u_{\lambda,\epsilon}+tM_n)^2}\;dx+C_2|\Omega|\nonumber\\
        &\leq \frac{1}{2}N_\epsilon^2(u_{\lambda,\epsilon})+\frac{t^2}{2}N_\epsilon^2(M_n)+tB_\epsilon(u_{\lambda,\epsilon},M_n)-C_1\int_\Omega e^{\alpha (u_{\lambda,\epsilon}+tM_n)^2}\;dx+C_2|\Omega|
    \end{align}
where $C_1,C_2>0$ are independent of $n$ and $\epsilon.$ Using the definition of $M_n$ in \ref{D00}, we obtain
\begin{align}\label{D0}
        E_{\lambda,\epsilon}(u_{\lambda,\epsilon}+tM_n)&\leq C(1+t+t^2)-C_1\int_{B(x_0,\frac{r}{n})} e^{\alpha (u_{\lambda,\epsilon}+tM_n)^2}\;dx\nonumber\\
         &\leq C(1+t+t^2)-C_1\int_{B(x_0,\frac{r}{n})} e^{\alpha (tM_n)^2}\;dx\nonumber\\
        &\leq C(1+t+t^2)-C_1n^{\frac{\alpha t^2}{2\pi}-2},
\end{align}
where the constants $C,C_1>0$ are independent of $n$ and $\epsilon.$ Hence, the lemma follows from the estimation (\ref{D0}).
\end{proof}

\begin{Lemma}\label{CL2}
    Suppose $\lambda\in(0,\Lambda)$ is any fixed real. Then, there exists $\epsilon_1>0$ satisfying the following property: for every $\epsilon\in(0,\epsilon_1)$, there exists $N=N(\epsilon)\in\N$ such that 
  $$\max_{t\geq 0}E_{\lambda,\epsilon}(u_{\lambda,\epsilon}+tM_N)<E_{\lambda,\epsilon}(u_{\lambda,\epsilon})+\frac{2\pi}{\alpha_0}.$$
\end{Lemma}
\begin{proof}
We prove it by contradiction. If the conclusion is not true, then there exists a decreasing sequence $\epsilon_m\to 0$ such that 
\begin{align*}
\max_{t\geq 0}E_{\lambda,\epsilon_m}(u_{\lambda,\epsilon_m}+tM_n)\geq E_{\lambda,\epsilon_m}(u_{\lambda,\epsilon_m})+\frac{2\pi}{\alpha_0}\text{ for all }n\in\N.
\end{align*}
One can easily prove that there exists $t_{nm}>0$ such that $$\max_{t\geq 0}E_{\lambda,\epsilon_m}(u_{\lambda,\epsilon_m}+tM_n)=E_{\lambda,\epsilon_m}(u_{\lambda,\epsilon_m}+t_{nm}M_n).$$ Hence, one has 
\begin{align}\label{D1}
    \frac{2\pi}{\alpha_0}\leq E_{\lambda,\epsilon_m}(u_{\lambda,\epsilon_m}+t_{nm}M_n)- E_{\lambda,\epsilon_m}(u_{\lambda,\epsilon_m})\text{ for all }n,\;m\in\N.
\end{align}
Furthermore, (\ref{D0}) guarantees that there exists $t^*>0$ such that $t_{nm}\leq t^*$ for all $n\;,m\in\N.$
We also have 
\begin{align*}
    \frac{d}{dt}E_{\lambda,\epsilon_m}(u_{\lambda,\epsilon_m}+tM_n)|_{t=t_{nm}}=0,
\end{align*}
which gives
\begin{align}\label{D5}
    \int_\Omega h(x, u_{\lambda,\epsilon_m+t_{nm}M_n})t_{nm}M_n\;dx&=t_{nm}^2N_{\epsilon_m}^2(M_n)+t_{nm}B_{\epsilon_m}(u_{\lambda,\epsilon_m},M_n)\nonumber\\
    &-\lambda t_{nm}\int_\Omega\frac{M_n}{(u_{\lambda,\epsilon_m+t_{nm}M_n})^{\delta(x)}}\;dx.
\end{align}
Thus, (\ref{D1}) yields
\begin{align}\label{D2}
    \frac{2\pi}{\alpha_0}&\leq E_{\lambda,\epsilon_m}(u_{\lambda,\epsilon_m}+t_{nm}M_n)-E_{\lambda,\epsilon_m}(u_{\lambda,\epsilon_m})\nonumber\\
    &\leq \frac{1}{2}\left(N_{\epsilon_m}^2(u_{\lambda,\epsilon_m}+t_{nm}M_n)-N_{\epsilon_m}^2(u_{\lambda,\epsilon_m})\right)-\lambda\int_\Omega \left(\frac{(u_{\lambda,\epsilon_m}+t_{nm}M_n)^{1-\delta(x)}-u_{\lambda,\epsilon_m}^{1-\delta(x)}}{1-\delta(x)} \right)\;dx\nonumber\\
    &-\int_\Omega\left( H(x,u_{\lambda,\epsilon_m}+t_{nm}M_n)-H(x,u_{\lambda,\epsilon_m})\right)\;dx\nonumber\\
    &\leq \frac{t^2_{{nm}}}{2}N_{\epsilon_m}^2(M_n)+B_{\epsilon_m}(u_{\lambda,\epsilon_m},t_{nm}M_n)-\lambda\int_\Omega \left(\frac{(u_{\lambda,\epsilon_m}+t_{nm}M_n)^{1-\delta(x)}-u_{\lambda,\epsilon_m}^{1-\delta(x)}}{1-\delta(x)} \right)\;dx\nonumber\\
    &-\int_\Omega\left( H(x,u_{\lambda,\epsilon_m}+t_{nm}M_n)-H(x,u_{\lambda,\epsilon_m})\right)\;dx\nonumber\\
    &\leq \frac{t^2_{nm}}{2}N_{\epsilon_m}^2(M_n)+\lambda\int_\Omega\left[\frac{t_{nm}M_n}{u_{\lambda,\epsilon_m}^{\delta(x)}}-\left(\frac{(u_{\lambda,\epsilon_m}+t_{nm}M_n)^{1-\delta(x)}-u_{\lambda,\epsilon_m}^{1-\delta(x)}}{1-\delta(x)} \right)\right]\;dx\nonumber\\
    &-\int_\Omega \left ( H(x,u_{\lambda,\epsilon_m}+t_{nm}M_n)-H(x,u_{\lambda,\epsilon_m})-h(x,u_{\lambda,\epsilon_m})t_{nm}M_n\right )\;dx
\end{align}
In the last inequality, we used the fact that $u_{\lambda,\epsilon_m}$ is a solution of (\ref{ME}). Using (H1), (H2) and (\ref{bdd}), \cite[Lemma 2.3]{Soares06} leads to conclude that
\begin{align}\label{D3}
    H(x,u_{\lambda,\epsilon_m}+t_{nm}M_n)-H(x,u_{\lambda,\epsilon_m})-h(x,u_{\lambda,\epsilon_m})t_{nm}M_n\geq -C(t_{nm}M_n)^2,
\end{align}
where $C>0$ is a constant independent of $n,m.$ Combining (\ref{D3}) and mean value theorem, (\ref{D2}) gives
\begin{align}\label{D4}
\frac{2\pi}{\alpha_0}&\leq \frac{t^2_{nm}}{2}N_{\epsilon_m}^2(M_n)+\lambda\int_\Omega\left[\frac{t_{nm}M_n}{u_{\lambda,\epsilon_m}^{\delta(x)}}-\frac{t_{nm}M_n}{(u_{\lambda,\epsilon_m}+\theta_1t_{nm}M_n)^{\delta(x)}}\right]\;dx+Ct_{nm}^2\int_\Omega M_n^2\;dx \nonumber\\
&\leq \frac{t^2_{nm}}{2}N_{\epsilon_m}^2(M_n)+\lambda\int_\Omega\frac{\delta(x)t_{nm}^2M_n^2}{(u_{\lambda,\epsilon_m}+\theta_2t_{nm}M_n)^{\delta(x)}}\;dx+Ct_{nm}^2\int_\Omega M_n^2\;dx\nonumber\\
&\leq \frac{t^2_{nm}}{2}N_{\epsilon_m}^2(M_n)+Ct_{nm}^2\int_\Omega M_n^2\;dx\nonumber\\
&\leq \frac{t_{nm}^2}{2}(1+C_1\epsilon_m)+C_2t_{nm}^2O(\frac{1}{\mathrm{log}(n)}),
\end{align}
where $C_1,\;C_2>0$ are two constants independent of $n,\;m$. In the second last inequality, we have used the fact (\ref{bdd}). From (\ref{D4}), we infer that $t_{nm}\geq t_*>0$ for all $n,\;m\in\N$. We also observe that
given any $n\in\N$ (with $\frac{1}{n}<\epsilon_0$, obtained in Lemma \ref{UB}), there exists $N(n)\in\N$ such that $\epsilon_m<\frac{1}{n}$ for all $m\geq N(n).$ Thus, for large enough $m\in\N$, (\ref{D4}) ensures that
\begin{align}\label{D50}
    \frac{\alpha_0t_{nm}^2}{4\pi}\geq 1-C_1\epsilon_m-C_2O(\frac{1}{\mathrm{log}(n)})\geq 1-\frac{C_1}{n}-C_2O(\frac{1}{\mathrm{log}(n)}).
\end{align}
Now, using \ref{bdd} and (H1), we obtain
\begin{align}\label{D6}
    \int_\Omega h(x, u_{\lambda,\epsilon_m}&+t_{nm}M_n)t_{nm}M_n\;dx\geq \int_{B(x_0,\frac{r}{n})} h(x, u_{\lambda,\epsilon_m}+t_{nm}M_n)t_{nm}M_n\;dx\nonumber\\
    &\geq \int_{B(x_0,\frac{r}{n})} h(x, \tau+t_{nm}(\frac{\mathrm{log}n}{2\pi})^\frac{1}{2})t_{nm}(\frac{\mathrm{log}n}{2\pi})^\frac{1}{2}\;dx\nonumber\\
    &=\int_{B(x_0,\frac{r}{n})} h(x, \tau+t_{nm}(\frac{\mathrm{log}n}{2\pi})^\frac{1}{2})\left(\tau+t_{nm}(\frac{\mathrm{log}n}{2\pi})^\frac{1}{2}\right)\left(\frac{t_{nm}(\frac{\mathrm{log}n}{2\pi})^\frac{1}{2}}{\tau+t_{nm}(\frac{\mathrm{log}n}{2\pi})^\frac{1}{2}}\right)\;dx.
\end{align}
Utilizing the facts $t_{nm}\geq t_*$, and (H6), we conclude that for large and fixed $n\in\N$ (to be determined later),
\begin{align}\label{D7}
 \int_\Omega h(x, u_{\lambda,\epsilon_m}&+t_{nm}M_n)t_{nm}M_n\;dx\geq \frac{1}{2}\int_{B(x_0,\frac{r}{n})} e^{\alpha_0\left(\tau+t_{nm}(\frac{\mathrm{log}n}{2\pi})^\frac{1}{2}\right)^2-\left(\tau+t_{nm}(\frac{\mathrm{log}n}{2\pi})^\frac{1}{2}\right)^\sigma}\;dx\nonumber\\
 &\geq \frac{Cr^2}{n^2}e^{\alpha_0\left(\tau+t_{nm}(\frac{\mathrm{log}n}{2\pi})^\frac{1}{2}\right)^2-\left(\tau+t_{nm}(\frac{\mathrm{log}n}{2\pi})^\frac{1}{2}\right)^\sigma}\nonumber\\
 &=Cr^2e^{\alpha_0\left(\tau+t_{nm}(\frac{\mathrm{log}n}{2\pi})^\frac{1}{2}\right)^2-\left(\tau+t_{nm}(\frac{\mathrm{log}n}{2\pi})^\frac{1}{2}\right)^\sigma-2\mathrm{log}n}\nonumber\\
 &\geq Ce^{\alpha_0t_{nm}(\frac{\mathrm{log}n}{2\pi})^\frac{1}{2}-\left(\tau+t_{nm}(\frac{\mathrm{log}n}{2\pi})^\frac{1}{2}\right)^\sigma}e^{2\left(\frac{\alpha_0t_{nm}^2}{4\pi}-1\right)\mathrm{log}n},
\end{align}
where $C>0$ is a constant independent of $n,m\in\N.$
Applying (\ref{D50}) in (\ref{D7}), we get
\begin{align}\label{D8}
     \int_\Omega h(x, u_{\lambda,\epsilon_m}&+t_{nm}M_n)t_{nm}M_n\;dx\geq Ce^{\alpha_0t_{nm}(\frac{\mathrm{log}n}{2\pi})^\frac{1}{2}-\left(\tau+t_{nm}(\frac{\mathrm{log}n}{2\pi})^\frac{1}{2}\right)^\sigma}e^{2\left(-\frac{C_1}{n}-C_2O(\frac{1}{\mathrm{log}n})\right)\mathrm{log}n}\nonumber\\
     &\geq C_3e^{\alpha_0t_{nm}(\frac{\mathrm{log}n}{2\pi})^\frac{1}{2}-\left(\tau+t_{nm}(\frac{\mathrm{log}n}{2\pi})^\frac{1}{2}\right)^\sigma},
\end{align}
for large $m\in\N,$ where $C_3>0$ is independent of $n,m.$ Since $t_{nm}$ is uniformly bounded, (\ref{D5}) leads to conclude that there exists a constant $M^*>0$ independent of $n,\;m$ such that 
\begin{align}\label{D9}
    \int_\Omega h(x, u_{\lambda,\epsilon_m}+t_{nm}M_n)t_{nm}M_n\;dx\leq M^*.
\end{align}
Since $\sigma<1$, we choose $n$ large enough such that
\begin{align}\label{D100}
    e^{\alpha_0t_{nm}(\frac{\mathrm{log}n}{2\pi})^\frac{1}{2}-\left(\tau+t_{nm}(\frac{\mathrm{log}n}{2\pi})^\frac{1}{2}\right)^\sigma}\geq \frac{2M^*}{C_3},
\end{align}
where the constant $C_3>0$ is obtained in (\ref{D8}). Combining (\ref{D8})-(\ref{D100}), we get a contradiction. This proves the lemma.
\end{proof}

\section{Proof of Theorem \ref{T2}}\label{PT2}
In order to apply Lemma \ref{CL2}, we assume $\epsilon\in(0,\epsilon_1)$ and $\lambda\in(0,\Lambda)$, where $\epsilon_1$ is obtained in Lemma \ref{CL2}. Before proceeding, we define  $$\mathcal{F}:=\{ u\in X: u\geq u_{\lambda,\epsilon} \text{ in }\Omega\}$$ and notice that the following Dichotomy holds:
\begin{enumerate}
    \item [($P_1$)]For every $r\in(0,r_2),$
     \begin{align*}
        \inf\{ E_{\lambda,\epsilon}(u): u\in\mathscr{F},\;\|u-u_{\lambda,\epsilon}\|=r\}=E_{\lambda,\epsilon}(u_{\lambda,\epsilon}) \text{ for every }r\in(0,r_2).
    \end{align*}
\item [($P_2$)]There exists $r\in(0,r_2)$ such that 
    \begin{align*}
        \inf\{ E_{\lambda,\epsilon}(u): u\in\mathscr{F},\;\|u-u_{\lambda,\epsilon}\|=r\}>E_{\lambda,\epsilon}(u_{\lambda,\epsilon}),
\end{align*}
\end{enumerate}
where $r_2$ and $u_{\lambda,\epsilon}$ are mentioned in the previous section.
In both the cases, we will prove the existence of another solution. \\
\textbf{Proof of Theorem \ref{T2}:}
We divide the proof into two parts.\\
\textbf{Case-I:} We suppose that  
    \begin{align*}
        \inf\{ E_{\lambda,\epsilon}(u): u\in\mathscr{F},\;\|u-u_{\lambda,\epsilon}\|=r\}=E_{\lambda,\epsilon}(u_{\lambda,\epsilon}) \text{ for every }r\in(0,r_2).
    \end{align*}
In this case, we prove that for small enough $r\in(0,\frac{r_2}{2})$, there exists a solution $v_{\lambda,\epsilon}\in X$ of (\ref{ME}) such that $\|u_{\lambda,\epsilon}-v_{\lambda,\epsilon}\|=r.$ To this aim, for fixed $r\in(0,r_2)$, we choose $t>0$ such that $0<r-t<r+t<2r<r_2$ and define
\begin{equation*}
    \mathscr{R}:=\{u\in \mathscr{F}: r-t\leq\|u_{\lambda,\epsilon}-u\|\leq r+t\},
\end{equation*}
which is closed in $X.$ Let $\{u_n\}\subset\mathscr{R}$ such that $\|u_{\lambda,\epsilon}-u\|=r$ and $E_{\lambda,\epsilon}(u_n)\to E_{\lambda,\epsilon}(u_{\lambda,\epsilon})$ as $n\to\infty.$ By Ekeland's variational principle, there exists a sequence $\{v_n\}\subset\mathscr{R}$ such that the following hold:
\begin{align}\label{D10}
    \begin{cases}
    E_{\lambda,\epsilon}(v_n)\leq E_{\lambda,\epsilon}(u_n)\leq E_{\lambda,\epsilon}(u_{\lambda,\epsilon})+\frac{1}{n},\\
    \|u_n-v_n\|\leq \frac{1}{n},\\
    E_{\lambda,\epsilon}(v_n)\leq E_{\lambda,\epsilon}(u)+\frac{1}{n}\|v_n-u\|\text{ for all }u\in\mathscr{R}.
\end{cases}
\end{align}
Since $\{v_n\}$ is a bounded sequence in $X,$ there exist a subsequence of $\{v_n\}$, still denoted by $\{v_n\}$ and $v_{\lambda,\epsilon}\in X$ such that $v_n\rightharpoonup v_{\lambda,\epsilon}$ weakly in $X$, $v_n\to v_{\lambda,\epsilon}$ strongly in $L^2(\Omega)$ and $v_n\to v_{\lambda,\epsilon}$ pointwise a.e. in $\Omega.$ We claim that $v_{\lambda,\epsilon}$ is a solution of the equation (\ref{ME}). To this end, suppose $\phi\in C^\infty_c(\Omega).$ For $\eta>0,$ define $\phi_n=(v_n+\eta\phi-u_{\lambda,\epsilon})^-$ and $$w_n=:v_n+\eta\phi+\phi_n=\max\{ v_n+\eta\phi,u_{\lambda,\epsilon}\}.$$
It is immediate that $w_n\in\mathscr{F}.$ Given any $\psi\in\mathscr{F}$, incorporating $u=v_n+\eta(\psi-v_n)$ in (\ref{D10}), we obtain
\begin{align}\label{D11}
-\frac{\eta}{n}\|v_n-\psi\|&\leq E_{\lambda,\epsilon}(v_n+\eta(\psi-v_n))-E_{\lambda,\epsilon}(v_n)\nonumber\\
&=\frac{1}{2}(N_\epsilon^2(v_n+\eta(\psi-v_n))-N_\epsilon^2(v_n))-\lambda\int_\Omega \frac{ \left( (v_n+\eta(\psi-v_n))^{1-\delta(x)}-v_n^{1-\delta(x)}\right )}{1-\delta(x)}\;dx\nonumber\\
&-\int_\Omega \left(H(x,v_n+\eta(\psi-v_n))-H(x,v_n)\right)\;dx\nonumber\\
&=\frac{1}{2}(N_\epsilon^2(v_n+\eta(\psi-v_n))-N_\epsilon^2(v_n))-\lambda\int_\Omega\frac{\eta(\psi-v_n)}{ (v_n+\eta\theta_1(\psi-v_n))^{\delta(x)}}\;dx\nonumber\\
&-\eta\int_\Omega h(x,v_n+\eta\theta_2(\psi-v_n))(\psi-v_n)\;dx,
\end{align}
where $\theta_1,\theta_2\in(0,1).$ Dividing both sides of (\ref{D11}) by $\eta$ and taking $\eta\to 0$, we get
\begin{align}\label{D12}
    -\frac{1}{n}\|\psi-v_n\|\leq B_\epsilon(v_n,\psi-v_n)-\lambda\int_\Omega\frac{\psi-v_n}{v_n^{\delta(x)}}\;dx-\int_\Omega h(x,v_n)(\psi-v_n)\;dx.
\end{align}
Putting $\psi=w_n$ in (\ref{D12}), one has
\begin{align}\label{D14}
    -\frac{1}{n}\|\eta\phi+\phi_n\|\leq B_\epsilon(v_n,\eta\phi+\phi_n)-\lambda\int_\Omega\frac{\eta\phi+\phi_n}{v_n^{\delta(x)}}\;dx-\int_\Omega h(x,v_n)(\eta\phi+\phi_n)\;dx.
\end{align}
By similar argument as in \cite[ page 37, equation 5.21]{BiagiCV}, we obtain
\begin{align}
    B_\epsilon(v_n,\eta\phi+\phi_n)\leq B_\epsilon(v_{\lambda,\epsilon},\eta\phi+\phi_\eta)+o(1),
\end{align}
where $\phi_\eta=(v_{\lambda,\epsilon}+\eta\phi-u_{\lambda,\epsilon})^-.$ By Lebesgue's dominated convergence theorem, we get
\begin{align}
   \int_\Omega\frac{\eta\phi+\phi_n}{v_n^{\delta(x)}}\;dx\to\int_\Omega\frac{\eta\phi+\phi_\eta}{v_{\lambda,\epsilon}^{\delta(x)}}\;dx \text{ as }n\to\infty,
\end{align}
and using Egoroff's theorem, we deduce
\begin{align}
    \int_\Omega h(x,v_n)(\eta\phi+\phi_n)\;dx\to \int_\Omega h(x,v_{\lambda,\epsilon})(\eta\phi+\phi_\eta)\;dx \text{ as }n\to\infty.
\end{align}
Using this facts, (\ref{D14}) leads to conclude that
\begin{align}\label{D13}
    \eta \left[ B_\epsilon(v_{\lambda,\epsilon},\phi)-\lambda\int_\Omega\frac{\phi}{v_{\lambda,\epsilon}^{\delta(x)}}\;dx-\int_\Omega h(x,v_{\lambda,\epsilon})\phi\;dx\right]&\geq  -B_\epsilon(v_{\lambda,\epsilon},\phi_\eta)+\lambda\int_\Omega\frac{\phi_\eta}{v_{\lambda,\epsilon}^{\delta(x)}}\;dx\nonumber\\
    &+\int_\Omega h(x,v_{\lambda,\epsilon})\phi_\eta\;dx.
\end{align}
Since $u_{\lambda,\epsilon}$ is a solution of (\ref{ME}), (\ref{D13}) yields
\begin{align}\label{D15}
     B_\epsilon(v_{\lambda,\epsilon},\phi)&-\lambda\int_\Omega\frac{\phi}{v_{\lambda,\epsilon}^{\delta(x)}}\;dx-\int_\Omega h(x,v_{\lambda,\epsilon})\phi\;dx
     \geq \frac{1}{\eta} B_\epsilon(u_{\lambda,\epsilon}-v_{\lambda,\epsilon},\phi_\eta)+\nonumber\\
     &\frac{1}{\eta}\left(\lambda\int_\Omega \left(\frac{1}{v_{\lambda,\epsilon}^{\delta(x)}}-\frac{1}{u_{\lambda,\epsilon}^{\delta(x)}}\right)\phi_\eta\;dx+\int_\Omega (h(x,v_{\lambda,\epsilon})-h(x,u_{\lambda,\epsilon}))\phi_\eta\;dx \right).
\end{align}
In order to show, R.H.S. is non-negative, we estimate the following quantities
\begin{align}
    \frac{1}{\eta} B_\epsilon(u_{\lambda,\epsilon}-v_{\lambda,\epsilon},\phi_\eta)&=\frac{1}{\eta} \int_{\Omega_\eta}\nabla(u_{\lambda,\epsilon}-v_{\lambda,\epsilon})\cdot\nabla(u_{\lambda,\epsilon}-v_{\lambda,\epsilon}-\eta\phi)\;dx\nonumber\\
    &+\frac{\epsilon}{\eta}\int_{\R^2}\int_{\R^2}\frac{((u_{\lambda,\epsilon}-v_{\lambda,\epsilon})(x)-(u_{\lambda,\epsilon}-v_{\lambda,\epsilon})(y))(\phi_\eta(x)-\phi_\eta(y))}{|x-y|^{2+2s}}\;dxdy\nonumber\\
    &\geq o(1)\text{ as }\eta\to 0.
\end{align}
Using mean value theorem and definition of $\phi_\eta$, we have
\begin{align}
\frac{\lambda}{\eta}\int_\Omega \left(\frac{1}{v_{\lambda,\epsilon}^{\delta(x)}}-\frac{1}{u_{\lambda,\epsilon}^{\delta(x)}}\right)\phi_\eta\;dx&=\frac{\lambda}{\eta}\int_{\Omega_\eta} \left(\frac{1}{v_{\lambda,\epsilon}^{\delta(x)}}-\frac{1}{u_{\lambda,\epsilon}^{\delta(x)}}\right)(u_{\lambda,\epsilon}-v_{\lambda,\epsilon})\;dx\nonumber\\
&-\lambda\int_{\Omega_\eta} \left(\frac{1}{v_{\lambda,\epsilon}^{\delta(x)}}-\frac{1}{u_{\lambda,\epsilon}^{\delta(x)}}\right)\phi\;dx\nonumber\\
&\geq \lambda \int_{\Omega_\eta}\frac{\delta(x)(v_{\lambda,\epsilon}-u_{\lambda,\epsilon})\phi}{(u_{\lambda,\epsilon}+\theta(v_{\lambda,\epsilon}-u_{\lambda,\epsilon}))^{1+\delta(x)}}\;dx=o(1)\text{ as }\eta\to 0.
\end{align}
We also have
\begin{align}\label{D16}
    \frac{1}{\eta}\int_\Omega (h(x,v_{\lambda,\epsilon})-h(x,u_{\lambda,\epsilon}))\phi_\eta\;dx&=\frac{1}{\eta}\int_{\Omega_\eta} (h(x,v_{\lambda,\epsilon})-
    h(x,u_{\lambda,\epsilon}))(u_{\lambda,\epsilon}-v_{\lambda,\epsilon})\;dx-\nonumber\\
    &\int_{\Omega_\eta} (h(x,v_{\lambda,\epsilon})-h(x,u_{\lambda,\epsilon}))\phi\;dx\nonumber\\
    &= -\frac{1}{\eta}\int_{\Omega_\eta} \frac {\partial h}{\partial s}(x,u_{\lambda,\epsilon}+\theta(v_{\lambda,\epsilon}-u_{\lambda,\epsilon}))(u_{\lambda,\epsilon}-v_{\lambda,\epsilon})^2\;dx-\nonumber\\
    &\int_{\Omega_\eta} (h(x,v_{\lambda,\epsilon})-h(x,u_{\lambda,\epsilon}))\phi\;dx,
\end{align}
for some $\theta\in(0,1)$. Utilizing the fact that $0\leq v_{\lambda,\epsilon}-u_{\lambda,\epsilon}\leq \eta|\phi|$ in $\Omega_\eta$, (\ref{D16}) implies
\begin{align}
    \frac{1}{\eta}\int_\Omega (h(x,v_{\lambda,\epsilon})-h(x,u_{\lambda,\epsilon}))\phi_\eta\;dx&\geq -\eta\int_{\Omega_\eta} \frac {\partial h}{\partial s}(x,u_{\lambda,\epsilon}+\theta(v_{\lambda,\epsilon}-u_{\lambda,\epsilon}))\phi^2\;dx-\nonumber\\
    &\int_{\Omega_\eta} (h(x,v_{\lambda,\epsilon})-h(x,u_{\lambda,\epsilon}))\phi\;dx=o(1) \text{ as }\eta\to 0
\end{align}
Taking $\eta\to 0$ in (\ref{D15}) and applying the above facts, we deduce
\begin{equation*}
    B_\epsilon(v_{\lambda,\epsilon},\phi)-\lambda\int_\Omega\frac{\phi}{v_{\lambda,\epsilon}^{\delta(x)}}\;dx-\int_\Omega h(x,v_{\lambda,\epsilon})\phi\;dx\geq 0 \text{ for all }\phi\in C^\infty_c(\Omega).
\end{equation*}
Hence, $v_{\lambda,\epsilon}\in X$ is a solution of equation (\ref{ME}). We claim $v_n\to v_{\lambda,\epsilon}$ in $X.$ To this aim, incorporating $\psi=v_{\lambda,\epsilon}$ in (\ref{D12}), one has
\begin{align*}
    -\frac{1}{n}\|v_n-v_{\lambda,\epsilon}\|&\leq B_\epsilon(v_n,v_{\lambda,\epsilon}-v_n)-\lambda\int_\Omega\frac{v_{\lambda,\epsilon}-v_n}{v_n^{\delta(x)}}\;dx-\int_\Omega h(x,v_n)(v_{\lambda,\epsilon}-v_n)\;dx\nonumber\\
    &\leq B_\epsilon(v_n-v_{\lambda,\epsilon},v_{\lambda,\epsilon}-v_n)+B_\epsilon(v_{\lambda,\epsilon},v_{\lambda,\epsilon}-v_n)-\lambda\int_\Omega\frac{v_{\lambda,\epsilon}-v_n}{v_n^{\delta(x)}}\;dx-\nonumber\\
    &\int_\Omega h(x,v_n)(v_{\lambda,\epsilon}-v_n)\;dx\nonumber\\
    &= -N_\epsilon^2(v_n-v_{\lambda,\epsilon})+B_\epsilon(v_{\lambda,\epsilon},v_{\lambda,\epsilon}-v_n)-\lambda\int_\Omega\frac{v_{\lambda,\epsilon}-v_n}{v_n^{\delta(x)}}\;dx\nonumber\\
    &-\int_\Omega h(x,v_n)(v_{\lambda,\epsilon}-v_n)\;dx,
\end{align*}
which implies
\begin{align}\label{D23}
    N_\epsilon^2(v_n-v_{\lambda,\epsilon})&\leq \frac{1}{n}\|v_n-v_{\lambda,\epsilon}\|+B_\epsilon(v_{\lambda,\epsilon},v_{\lambda,\epsilon}-v_n)-\lambda\int_\Omega\frac{v_{\lambda,\epsilon}-v_n}{v_n^{\delta(x)}}\;dx\nonumber\\
    &-\int_\Omega h(x,v_n)(v_{\lambda,\epsilon}-v_n)\;dx\nonumber\\
    &=o(1)+\lambda\int_\Omega\frac{v_n-v_{\lambda,\epsilon}}{v_n^{\delta(x)}}\;dx+\int_\Omega h(x,v_n)(v_n-v_{\lambda,\epsilon})\;dx.
\end{align}
Since $\frac{v_{\lambda,\epsilon}}{v_n^{\delta(x)}}\leq \frac{v_{\lambda,\epsilon}}{u_{\lambda,\epsilon}^{\delta(x)}}\in L^1(\Omega)$, by Lebesgue's dominated convergence theorem, we have
\begin{align}\label{D21}
    \int_\Omega\frac{v_{\lambda,\epsilon}}{v_n^{\delta(x)}}\;dx\to\int_\Omega v_{\lambda,\epsilon}^{1-\delta(x)}\;dx.
\end{align}
Furthermore, 
\begin{align}\label{D22}
    \lim{n\to\infty}\int_\Omega v_n^{1-\delta(x)}\;dx=\int_\Omega v_{\lambda,\epsilon}^{1-\delta(x)}\;dx.
\end{align}
We will prove that 
\begin{equation}\label{D20}
    \int_\Omega h(x,v_n)(v_n-v_{\lambda,\epsilon})\;dx\to 0 \text{ as }n\to\infty.
\end{equation}
In order to prove this fact, we use Egoroff's theorem. Let $\xi>0$ be any arbitrary real number. Since $h(x,v_n)(v_n-v_{\lambda,\epsilon})\to 0$ pointwise a.e. in $\Omega$, Egoroff's theorem guarantees that there exists $E\subset\Omega$, such that $|E|<\xi$ and $$h(x,v_n)(v_n-v_{\lambda,\epsilon})\to 0\text{ uniformly in } \Omega\setminus E.$$
There exists $N_0(\xi)\in\N$ suc that $$|h(x,v_n)(v_n-v_{\lambda,\epsilon})|\leq \frac{\xi}{2|\Omega\setminus E|}\text{ for all }n\geq N_0.$$
Thus, for $n\geq N_0$, we have
\begin{align}\label{D19}
    \int_\Omega|h(x,v_n)(v_n-v_{\lambda,\epsilon})|\;dx&\leq \int_{\Omega\setminus E}|h(x,v_n)(v_n-v_{\lambda,\epsilon})|\;dx+\int_E |h(x,v_n)(v_n-v_{\lambda,\epsilon})|\;dx\nonumber\\
    &\leq \frac{\xi}{2|\Omega\setminus E|}|\Omega\setminus E|+|E|^\frac{1}{2}\left(\int_E |h(x,v_n)(v_n-v_{\lambda,\epsilon})|^2\;dx\right)^\frac{1}{2}\nonumber\\
    &\leq \frac{\xi}{2}+\xi^2\left(\int_\Omega |h(x,v_n)(v_n-v_{\lambda,\epsilon})|^2\;dx\right)^\frac{1}{2}.
\end{align}
Now, for a fixed $\alpha>\alpha_0$, there exists $C>0$ such that $h(x,s)\leq Ce^{\alpha s^2}$ for all $s>0$ and for all $x\in\Omega.$ Using this fact, generalized H\"{o}lder inequality and Lemma \ref{Emb}, we get
\begin{align}\label{D17}
    \int_\Omega |h(x,v_n)(v_n-v_{\lambda,\epsilon})|^2\;dx&\leq C\int_\Omega e^{2\alpha v_n^2}|(v_n-v_{\lambda,\epsilon})|^2\;dx\nonumber\\
    &\leq C\int_\Omega e^{4\alpha (v_n-u_{\lambda,\epsilon})^2}e^{4\alpha u_{\lambda,\epsilon}^2}|(v_n-v_{\lambda,\epsilon})|^2\;dx\nonumber\\
    &\leq C\left(\int_\Omega e^{4p\alpha (v_n-u_{\lambda,\epsilon})^2}\;dx\right)^\frac{1}{p}\left(\int_\Omega e^{4q\alpha u_{\lambda,\epsilon}^2}\;dx\right)^\frac{1}{q}\left(\int_\Omega |v_n-v_{\lambda,\epsilon}|^{2l}\;dx\right)^\frac{1}{l}\nonumber\\
    &\leq C \left(\int_\Omega e^{4p\alpha \|v_n-u_{\lambda,\epsilon}\|^2\left(\frac{v_n-u_{\lambda,\epsilon}}{\|v_n-u_{\lambda,\epsilon}\|}\right)^2}\;dx\right)^\frac{1}{p}\nonumber\\
    &\leq C\left(\int_\Omega e^{4p\alpha 
    (r+t)^2\left(\frac{v_n-u_{\lambda,\epsilon}}{\|v_n-u_{\lambda,\epsilon}\|}\right)^2}\;dx\right)^\frac{1}{p}\nonumber\\
    &\leq C\left(\int_\Omega e^{16p\alpha 
    r^2\left(\frac{v_n-u_{\lambda,\epsilon}}{\|v_n-u_{\lambda,\epsilon}\|}\right)^2}\;dx\right)^\frac{1}{p}, 
\end{align}
where $\frac{1}{p}+\frac{1}{q}+\frac{1}{l}=1$ and $r+t<2r.$ Choose $r>0$ small enough such that $p\alpha r^2<\frac{\pi}{4}.$ Hence, applying Theorem \ref{Moser} in (\ref{D17}), we get
\begin{align}\label{D18}
    \int_\Omega |h(x,v_n)(v_n-v_{\lambda,\epsilon})|^2\;dx\leq C,
\end{align}
where the constant $C>0$ is independent of $n$. Therefore, (\ref{D19}) and (\ref{D18}) ensure
$$\int_\Omega h(x,v_n)(v_n-v_{\lambda,\epsilon})\;dx\to 0 \text{ as }n\to\infty.$$ Combining (\ref{D21}), (\ref{D22}), (\ref{D20}); the inequality (\ref{D23}) concludes that $\|v_n-v_{\lambda,\epsilon}\|\to 0\text{ as }n\to\infty.$ Consequently, $\|v_{\lambda,\epsilon}-u_{\lambda,\epsilon}\|=r.$ By Lemma \ref{Comp}, one has $0<u_{\lambda,\epsilon}<v_{\lambda,\epsilon}$. This completes Case-I.

\textbf{Case-II:} We assume there exists $r\in(0,r_2)$ such that 
    \begin{align*}
        \inf\{ E_{\lambda,\epsilon}(u): u\in\mathscr{F},\;\|u-u_{\lambda,\epsilon}\|=r\}>E_{\lambda,\epsilon}(u_{\lambda,\epsilon}).
\end{align*}
In order to show the existence of another solution, we consider 
\begin{equation*}
    \Gamma:=\{\gamma\in C([0,1],\mathscr{F}): \gamma(0)=u_{\lambda,\epsilon}, \|u_{\lambda,\epsilon}-\gamma(1)\|>r, E_{\lambda,\epsilon}(\gamma(1))<E_{\lambda,\epsilon}(u_{\lambda,\epsilon})\}.
\end{equation*}
Due to Lemma \ref{CL1}, $\Gamma\neq\emptyset.$ We define $$c_{\lambda,\epsilon}:=\inf_{\gamma\in\Gamma}\max_{t\in[0,1]}E_{\lambda,\epsilon}(\gamma(t)).$$
Similar way as in \cite[ inequality 5.33, page 41]{BiagiCV}, there exists a sequence $\{v_n\}_{n\in\N}\subset\mathscr{F}$ such that 
\begin{align}\label{G0}
   c_{\lambda,\epsilon}\leq E_{\lambda,\epsilon}(v_n)\leq c_{\lambda,\epsilon}+\frac{1}{n}
\end{align}
and
\begin{align}\label{G1}
   -\frac{C}{n}(1+N_\epsilon(u))\leq B_\epsilon(v_n,u-v_n)-\lambda\int_\Omega \frac{u-v_n}{v_n^{\delta(x)}}\;dx-\int_\Omega h(x,v_n)(u-v_n)\;dx \text{ for all }u\in\mathscr{F},
\end{align}
for some constant $C>0$ independent of $n$. Incorporating $u=2v_n$ in (\ref{G1}) we get
\begin{align}\label{G2}
    -\frac{C}{n}(1+N_\epsilon(v_n))\leq B_\epsilon(v_n,v_n)-\lambda\int_\Omega v_n^{1-\delta(x)}\;dx-\int_\Omega h(x,v_n)v_n\;dx.
\end{align}
Therefore, using (\ref{G0})-(\ref{G2}) and (\ref{G3}), for $\theta>2$ we have
\begin{align}
\left(\frac{1}{2}-\frac{1}{\theta}\right)\int_\Omega &h(x,v_n)v_n\;dx-\lambda\int_\Omega \left(\frac{1}{1-\delta(x)}-\frac{1}{2}\right)v_n^{1-\delta(x)}-M_1\nonumber\\
&=\frac{1}{2}\left(\int_\Omega h(x,v_n)v_n\;dx+\lambda\int_\Omega v_n^{1-\delta(x)}\;dx\right)-\nonumber\\
&\left(\frac{1}{\theta}\int_\Omega h(x,v_n)v_n\;dx+\lambda\int_\Omega\frac{v_n^{1-\delta(x)}}{1-\delta(x)}\;dx+M_1\right)\nonumber\\
&\leq \frac{1}{2}N_\epsilon^2(v_n)+\frac{C}{n}(1+N_\epsilon(v_n))-\left(\int_\Omega H(x,v_n)\;dx+\lambda\int_\Omega\frac{v_n^{1-\delta(x)}}{1-\delta(x)}\;dx\right)\nonumber\\
&\leq \frac{1}{2}N_\epsilon^2(v_n)+\frac{C}{n}(1+N_\epsilon(v_n))-\frac{1}{2}N_\epsilon^2(v_n)+c_{\lambda,\epsilon}+\frac{1}{n}\nonumber\\
&\leq o(1)N_\epsilon(v_n)+c_{\lambda,\epsilon}.
\end{align}
Thus, we deduce
\begin{align}
    \int_\Omega h(x,v_n)v_n\;dx\leq C\int_\Omega v_n^{1-\delta(x)}\;dx+o(1)N_\epsilon(v_n)+c_{\lambda,\epsilon},
\end{align}
which together with (\ref{G3}) imply
\begin{align}\label{G4}
    \int_\Omega H(x,v_n)\;dx\leq C\int_\Omega v_n^{1-\delta(x)}\;dx+o(1)N_\epsilon(v_n)+c_{\lambda,\epsilon}+M_1.
\end{align}

Combining (\ref{G0}), (\ref{G4}) and Lemma \ref{Emb}, one has
\begin{align}
     \frac{1}{2}N_\epsilon^2(v_n)&\leq C\int_\Omega v_n^{1-\delta(x)}\;dx+o(1)N_\epsilon(v_n)+2c_{\lambda,\epsilon}+M_1+\lambda\int_\Omega \frac{v_n^{1-\delta(x)}}{1-\delta(x)}\;dx+\frac{1}{n}\nonumber\\
    &\leq C(1+\|v_n\|+c_{\lambda,\epsilon})+o(1)N_\epsilon(v_n).
\end{align}
Consequently, the sequence $\{v_n\}$ is bounded in $X.$ There exist a subsequence of $\{v_n\}$, still denoted by $\{v_n\}$ and $v_{\lambda,\epsilon}\in X$ such that $v_n\to v_{\lambda,\epsilon}$ weakly in $X$, $v_n\to v_n$ pointwise a.e. in $\Omega$ and $v_n\to v_{\lambda,\epsilon}$ strongly in $L^q$ for every $1\leq q<\infty.$ Along the lines of proof of Case-I, one can prove $v_{\lambda,\epsilon}\in X$ is a solution of equation (\ref{ME}). The rest of the proof aims to conclude that $u_{\lambda,\epsilon}\neq v_{\lambda,\epsilon}$. It is enough to prove that $E_{\lambda,\epsilon}(u_{\lambda,\epsilon})\neq E_{\lambda,\epsilon}(v_{\lambda,\epsilon}).$ We prove it via contradiction. If possible, let $E_{\lambda,\epsilon}(u_{\lambda,\epsilon})=E_{\lambda,\epsilon}(v_{\lambda,\epsilon})$. In order to get a contradiction, we will show $v_n\to v_{\lambda,\epsilon}$ strongly in $X.$ To this end, we incorporate $u=v_{\lambda,\epsilon}$ in (\ref{G1}) and use boundedness of $\{v_n\}$ to get
\begin{align}
    N_\epsilon^2(v_n-v_{\lambda,\epsilon})\leq o(1)+\lambda\int_\Omega\frac{v_n-v_{\lambda,\epsilon}}{v_n^{\delta(x)}}\;dx+\int_\Omega h(x,v_n)(v_n-v_{\lambda,\epsilon})\;dx.
\end{align}
Utilizing (\ref{D21}) and (\ref{D22}), we deduce
\begin{align}\label{G000}
     N_\epsilon^2(v_n-v_{\lambda,\epsilon})\leq o(1)+\int_\Omega h(x,v_n)(v_n-v_{\lambda,\epsilon})\;dx.
\end{align}
We will prove that
\begin{align}\label{G00}
    \int_\Omega h(x,v_n)(v_n-v_{\lambda,\epsilon})\;dx\to 0 \text{ as }n\to\infty.
\end{align}
It follows from the facts (\ref{G4}) and boundedness of $\{v_n\}$ that
 \begin{equation}
     \lim_{n\to\infty}\int_\Omega H(x,v_n)\;dx=\int_\Omega H(x,v_{\lambda,\epsilon})\;dx.
 \end{equation}
Thus, (\ref{G0}) ensures that
\begin{equation}
\lim_{n\to\infty}N_\epsilon^2(v_n)=\underbrace{2\left(c_{\lambda,\epsilon}+\lambda\int_\Omega\frac{v_{\lambda,\epsilon}^{1-\delta(x)}}{1-\delta(x)}\;dx+\int_\Omega H(x,v_{\lambda,\epsilon})\;dx\right)}_{=L_{\lambda,\epsilon}}.
\end{equation}
Let $w_n=\frac{v_n}{N_\epsilon(v_n)}$ and $w=\frac{v_{\lambda,\epsilon}}{\sqrt{L_{\lambda,\epsilon}}}$. We have $w_n\rightharpoonup w$ weakly in $X.$ Since $0<\epsilon<\epsilon_1$, from Lemma \ref{CL2}, we get
\begin{equation*}
    c_{\lambda,\epsilon}-E_{\lambda,\epsilon}(u_{\lambda,\epsilon})<\frac{2\pi}{\alpha_0}.
\end{equation*}
 We choose $\beta>\alpha_0$ such that 
\begin{equation}
     c_{\lambda,\epsilon}-E_{\lambda,\epsilon}(v_{\lambda,\epsilon})=c_{\lambda,\epsilon}-E_{\lambda,\epsilon}(u_{\lambda,\epsilon})<\frac{2\pi}{\beta}<\frac{2\pi}{\alpha_0}.
\end{equation}
Since $\frac{\beta}{4\pi}<\frac{1}{2(c_{\lambda,\epsilon}-E_{\lambda,\epsilon}(u_{\lambda,\epsilon}))},$ we can choose $p>1$ such that 
\begin{equation}
    \frac{\beta}{4\pi}N_\epsilon^2(v_n)<p<\left(1-\frac{N_\epsilon^2(v_{\lambda,\epsilon})}{L_{\lambda,\epsilon}}\right)^{-1}\text{ for large }n.
\end{equation}
Since $\beta>\alpha_0$, by property (H2), there exists $q>1$ and $C>0$ (independent of $n$) such that
\begin{align}\label{G5}
    \int_\Omega |h(x,v_n)|^q\;dx\leq C\int_\Omega e^{\beta v_n^2}\;dx.
\end{align}
Due to \cite[Theorem 1.6]{PL}, we have
\begin{align}\label{G6}
    \int_\Omega e^{\beta v_n^2}\;dx\leq \int_\Omega e^{4\pi\left(\frac{\beta}{4\pi}N_\epsilon^2(v_n)\right)\left(\frac{v_n}{N_\epsilon(v_n)}\right)^2}\;dx\leq \int_\Omega e^{4\pi p\left(\frac{v_n}{N_\epsilon(v_n)}\right)^2}\;dx\leq C,
\end{align}
where $C>0$ is a constant independent of $n.$
Now, combining (\ref{G5}), (\ref{G6}), we deduce
\begin{align}
    \int_\Omega |h(x,v_n)(v_n-v_{\lambda,\epsilon})|^l\;dx\leq \|v_n-v_{\lambda,\epsilon}\|_{L^\frac{ql}{q-l}(\Omega)}^l\left(\int_\Omega |h(x,v_n)|^q\;dx\right)^\frac{l}{q}\leq C,
\end{align}
for $1<l<q$. Thus, by Egoroff's theorem, one can easily prove the fact (\ref{G00}). Consequently, (\ref{G000}) reveals $v_n\to v_{\lambda,\epsilon}$ strongly in $X.$ Since $E_{\lambda,\epsilon}$ is continuous so $E_{\lambda,\epsilon}(v_{\lambda,\epsilon})=c_{\lambda,\epsilon}$, which is a contradiction as $E_{\lambda,\epsilon}(u_{\lambda,\epsilon})=E_{\lambda,\epsilon}(v_{\lambda,\epsilon}).$ By Lemma \ref{Comp}, one has $0<u_{\lambda,\epsilon}<v_{\lambda,\epsilon}$ in $\Omega$. Hence, $u_{\lambda,\epsilon},\; v_{\lambda,\epsilon}\in X$ are two distinct positive solutions of equation (\ref{ME}).

\section*{Acknowledgment}
I would like to thank Prof. Kaushik Bal for several useful discussions and his invaluable suggestions.

\nocite{*}
\bibliographystyle{plain}
\bibliography{Reference.bib}
\end{document}